\documentclass[a4paper,12pt]{amsart}
\usepackage{amsthm, amsfonts, amssymb, mathtools, color} \usepackage[all]{xy}
\usepackage{fullpage} \usepackage[mathscr]{eucal}
\usepackage[normalem]{ulem}
\def\til#1{\widetilde{#1}}
\def\udl#1{\underline{#1}}
\def\ovl#1{\overline{#1}}
\def\dl{\langle\kern-.22em\langle} \def\dr{\rangle\kern-.22em\rangle}
\def\dal{\langle\kern-.22em\langle}
\def\dar{\rangle\kern-.22em\rangle} \def\dbl{[\kern-.12em[}
\def\dbr{]\kern-.12em]} \def\dpl{(\kern-.2em(} \def\dpr{)\kern-.2em)}
\def\dast{\ast\kern-.03em\ast}
\def\tast{\ast\kern-.26em\ast\kern-.26em\ast}
\def\longhookrightarrow{\lhook\joinrel\longrightarrow}

\def\bfmath#1{\textrm{{\boldmath $#1$}}}
\def\NN{\mathbb{N}} \def\ZZ{\mathbb{Z}} \def\QQ{\mathbb{Q}}
\def\RR{\mathbb{R}} \def\CC{\mathbb{C}} \def\FF{\mathbb{F}}

 \def\OO{\mathcal{O}} \def\mm{\mathfrak{m}}
\def\nn{\mathfrak{n}} \def\pp{\mathfrak{p}} \def\qq{\mathfrak{q}}
\def\alg{{\mathrm{alg}}} \def\tran{{\mathrm{tran}}}
\def\e{\varepsilon}\def\dt{\delta}
\def\cA{{\mathcal{A}}}\def\cB{{\mathcal{B}}}
\def\fR#1{\OO\dbl#1\dbr} \def\rR#1{\OO\dpl#1\dpr}
\def\fZ#1{\ZZ\dbl#1\dbr}  \def\ZR#1{\langle
  #1\rangle} \def\rig{\mathrm{rig}} \def\sp{\mathrm{sp}}
\def\cl{\mathrm{cl}} \def\an{\mathrm{an}}

\DeclareMathOperator{\Spec}{Spec}
\DeclareMathOperator{\Sp}{Sp}
\DeclareMathOperator{\Spf}{Spf}
\DeclareMathOperator{\Spm}{Spm}
\DeclareMathOperator{\Proj}{Proj}
\DeclareMathOperator{\Frac}{Frac}

\DeclareMathOperator{\Cont}{Cont}
\DeclareMathOperator{\Sym}{Sym}
\DeclareMathOperator{\Ed}{E}
\DeclareMathOperator{\Ve}{V}
\DeclareMathOperator{\mcd}{mcd}

\DeclareMathOperator{\red}{red}
\def\Af{\mathbf{Af}} \def\Ad{\mathbf{Ad}}
\def\udp#1{#1^{\natural}}
\def\gdp#1#2{#1^{\natural#2}}
\title{Rational points of rigid analytic sets: a Pila-Wilkie type
  theorem} \author{Gal Binyamini}\author{Fumiharu Kato} \thanks{The
  first author was supported by the ISRAEL SCIENCE FOUNDATION (grant
  No. 1167/17). This project has received funding from the European
  Research Council (ERC) under the European Union's Horizon 2020
  research and innovation programme (grant agreement No 802107). The
  second author was supported by JSPS KAKENHI Grant Number 17H02832.}
\begin{document}
\theoremstyle{plain} \newtheorem{thm}[subsubsection]{Theorem}
\newtheorem{prop}[subsubsection]{Proposition}
\newtheorem{lem}[subsubsection]{Lemma}
\newtheorem{cor}[subsubsection]{Corollary}
\newtheorem{probs}[subsubsection]{Problems}
\newtheorem{cla}{Claim}[subsubsection]
\theoremstyle{definition} \newtheorem{dfn}[subsubsection]{Definition}
\newtheorem{ntn}[subsubsection]{Notation}
\newtheorem{exa}[subsubsection]{Example}
\newtheorem{exas}[subsubsection]{Examples}
\newtheorem{assum}[subsubsection]{Assumption}
\newtheorem{sit}[subsubsection]{Situation}
\theoremstyle{remark} \newtheorem{rem}[subsubsection]{Remark}
\newtheorem{note}[subsubsection]{Note}

\begin{abstract}
  We establish a rigid analytic analog of the Pila-Wilkie counting
  theorem, giving subpolynomial upper bounds for the number of
  rational points in the transcendental part of a $\QQ_p$-analytic set,
  and the number of rational functions in a $\FF_q\dpl t\dpr$-analytic
  set.  For $\ZZ\dpl t\dpr$-analytic sets we prove such bounds
  uniformly for the specialization to every non-archimedean local
  field.
\end{abstract}


\maketitle


\setcounter{section}{0}
\section{Introduction}\label{sec-intro}
Our goal in this paper is to introduce a non-archimedean version of
the Pila-Wilkie counting theorem which applies uniformly to analytic
sets over local fields in the mixed-characteristic case $F=\QQ_p$ and
equicharacteristic case $F=\FF_q\dpl t \dpr$. We begin by recalling
the classical statement of the Pila-Wilkie counting theorem.

\subsection{Classical Pila-Wilkie}
For $X\subset\RR^n$, we define the \emph{algebraic part} $X^\alg$ to be
the union of all connected, positive-dimensional semialgebraic sets
contained in $X$. We set $X^\tran:=X\setminus X^\alg$. For $x\in\QQ^n$
we denote $H(x):=\max_iH(x_i)$, where $H(x_i)$ is the height of
$x_i$. We set
$$
  X(\QQ,H) := \{x\in X\cap\QQ^n \mid H(x)\le H\}.
$$

With these notations, Pila and Wilkie \cite{pw} proved the
following counting theorem.

\begin{thm}\label{thm:pw}
  Let $X$ be definable in an o-minimal structure. Then for every
  $\e>0$ there exists a constant $C(X,\e)$ such that
$$
    \#X^\tran(\QQ,H) \le C(X,\e)\cdot H^\e.
$$
\end{thm}

We refer the unfamiliar reader to \cite{scanlon:survey} for an
introduction to the notion of o-minimal
structures. Theorem~\ref{thm:pw} has found numerous applications in an
area of arithmetic geometry known as ``unlikely intersection
problems'' following a general strategy of Pila and Zannier. We also
refer the reader to \cite{scanlon:survey} for a survey of this
direction.

\subsection{Our main result}
We consider a rigid analytic analog of the Pila-Wilkie theorem, where
definable sets are replaced by the spectra of \emph{affinoid}
algebras.  For the precise definitions see
Section~\ref{sec-rigidgeom}.  For the purposes of this introduction,
we first introduce two types of affinoid algebras over local fields.

Set either $F=\QQ_p$ and $V=\ZZ_p$, or $F=\FF_q\dpl t\dpr$ and
$V=\FF_q\dbl t\dbr$. An affinoid $F$-algebra is an $F$-algebra of the
form $\cB:=B[1/p]$ (resp.\ $\cB:=B[1/t]$) where $B$ is a quotient
$V\dal x_1,\ldots,x_n\dar/\mathfrak{b}$ of a formal power series ring
(converging in the $p$-adic or $t$-adic topologies). 
The set of $F$-rational points of the rigid analytic spectrum of such an affinoid algebra
can naturally be interpreted as a subset of the unit polydisc $V^n$ in
$F^n$. We will be concerned with counting \emph{rational} points in
such a spectrum. In particular, for $F=\QQ_p$ we set
$$
(\Sp \cB)(\QQ,H) := \{ x\in\Sp \cB\cap\QQ^n \mid \max_i H(x_i) \le H\}
$$
as before, and for $F=\FF_q\dpl t\dpr$ we set
$$
(\Sp \cB)(\FF_q(t),H) := \{ x\in\Sp \cB\cap\FF_q(t)^n \mid \max_i \deg_t(x_i) \le \log_q H \}.
$$
We say that an irreducible closed subset of
$\Sp F\dal x_1,\ldots,x_n\dar$ is \emph{algebraic} if it is an
irreducible component of a closed set defined by a \emph{polynomial
  ideal} $I\subset F\dal x_1,\ldots,x_n\dar$, i.e. an ideal obtained
by extension from $F[x_1,\ldots,x_n]$. More generally we say that a
closed subset of $\Sp F\dal x_1,\ldots,x_n\dar$ is algebraic if its
irreducible components are algebraic. Finally we define
$(\Sp\cB)^\alg$ to be the union of all positive-dimensional closed
algebraic subsets of $\Sp\cB$, and $(\Sp\cB)^\tran$ to be its complement in $\Sp\cB$.
With these definitions in place, our
first analog of the Pila-Wilkie theorem is as follows.

\begin{thm}\label{thm:main-L}
  Let $\cB$ be an affinoid $F$-algebra as above. Then for every $\e>0$
  there exists a constant $C(\cB,\e)$ such that
$$
\#(\Sp\cB)^\tran(F_0,H) \le C(\cB,\e)\cdot H^\e, \qquad F_0=
\begin{cases}
  \QQ & (F=\QQ_p), \\
  \FF_q(t) & (F=\FF_q\dpl t\dpr).
\end{cases}
$$
\end{thm}

In the $\QQ_p$ case, our result is a special case of the result of
\cite{ccl:pw}, who proved a similar result more generally for
\emph{subanalytic} $\QQ_p$ sets. In the $\FF_q\dpl t\dpr$ case the
result is new.

\begin{rem}
  In \cite{demangos:pw,demangos:pw2}, Demangos claimed a proof of
  Theorem~\ref{thm:main-L} in both the $\QQ_p$ and the
  $\FF_q\dpl t\dpr$ cases. However it has been known to the experts
  for several years that the proof contains a substantial gap in both
  cases.
\end{rem}

Our main result is a uniform version of Theorem~\ref{thm:main-L} where
we allow the field $F$ to vary. Toward this end, we let $\OO$ denote
the ring of integer of an algebraic number field and replace $F$ by
$\rR t$ and $V$ by $\fR t$. We develop an analogous notion of affinoid
algebras in this context. For the purpose of this introduction, the
reader may consider an affinoid algebra to be an $\rR t$-algebra
$\cB:=B[1/t]$ where $B$ is a quotient
$B=\fR{t}\dal x_1,\ldots,x_n\dar/\mathfrak{b}$ of a power series ring
converging in the $t$-adic topology. 
Consider some local
non-archimedean field $F_\alpha$ with valuation ring $V_\alpha$, and an adic homomorphism $\alpha:\fR t\to V_\alpha$.
Let $\pi$ be a uniformizer of $V_\alpha$, $q_{\alpha}=p^f$ the number of elements of the residue field $V_\alpha/(\pi)$, and $r_{\alpha}\geq 1$ be the number given by $\alpha(t)=u\pi^{r_{\alpha}}$, where $u\in V^{\times}_{\alpha}$.
Then
the fiber
$(\Sp\cB)_\alpha$ is an $F_\alpha$-affinoid space. 
We set
$$
\sigma_{\alpha}=\begin{cases}[F_{\alpha}:\QQ_p]&(F_{\alpha}\supset\QQ_p),\\ r_{\alpha}&(F_{\alpha}\supset\FF_q\dpl t\dpr).\end{cases}
$$
Our main theorem is as follows.

\begin{thm}\label{thm:main-uniform}
  Let $\cB$ be an affinoid $\rR t$-algebra. Then for every $\e>0$
  there exists a constant $C(\cB,\e,\sigma_{\alpha})$
such that for every fiber
  $(\Sp\cB)_\alpha$ as above,
$$
\#(\Sp\cB)_\alpha^\tran(F_{\alpha,0},H) \le C(\cB,\e,\sigma_{\alpha})\cdot q_\alpha^n\cdot H^\e, \qquad F_{\alpha,0}=
\begin{cases}
  \QQ & (F_\alpha\supset\QQ_p), \\
  \FF_q(t) & (F_\alpha\supset\FF_q\dpl t\dpr).
\end{cases}
$$
\end{thm}

Theorem~\ref{thm:main-L} follows immediately from
Theorem~\ref{thm:main-uniform}. A similar result, which holds more
generally for $\fR t$-subanalytic sets, was established in
\cite[Theorem~B]{cfl:uniform-pw}.  However, the result in
\cite{cfl:uniform-pw} holds for local fields with residue
characteristic larger than some $N=N(X,\e)$.

More generally, consider a morphism of affinoid algebras $\cA\rightarrow\cB$,
and view $\Sp\cB\to\Sp\cA$ as a family of rigid analytic spaces. Then
any adic homomorphism $\alpha:A\to V_\alpha$ from a formal model $A$ of $\cA$ corresponds to a point of
$\Sp\cA$, and we have the following.

\begin{thm}\label{thm:main-uniform-family}
  Let $\cA,\cB$ be an affinoid $\rR t$-algebra. Then for any $\e>0$ and any positive integer $\sigma$
  there exists a constant $C(\cA,\cB,\e,\sigma)$ such that for every fiber
  $(\Sp\cB)_\alpha$ as above,
$$
\#(\Sp\cB)_\alpha^\tran(F_{\alpha,0},H) \le C(\cA,\cB,\e,\sigma_{\alpha})\cdot q_\alpha^n\cdot H^\e, \qquad F_{\alpha,0}=
\begin{cases}
  \QQ & (F_\alpha\supset\QQ_p), \\
  \FF_q(t) & (F_\alpha\supset\FF_q\dpl t\dpr).
\end{cases}
$$
\end{thm}

Theorem~\ref{thm:main-uniform-family} restricts to
Theorem~\ref{thm:main-uniform} where $\cA=\rR t$.

\subsection{Overconvergent versions}
Most of the technical work in the present paper is carried out with
the flavor of ``overconvergence''. That is, in the notations of the
previous section, we consider a formal model $B$ and affinoid algebra
$\cB:=B[1/t]$, but restrict attention to the part of $\Sp\cB$ that
belongs to the polydisc of radius $|t^\delta|$ for some $\delta\in\QQ$
with $\delta>0$.  Denote this by $\Sp \gdp{\cB}\delta$.  For this
overconvergent part, we have the following more uniform version of
Theorem~\ref{thm:main-uniform}.

\begin{thm}\label{thm:main-uniform-oc}
  Let $\cB$ be an affinoid $\rR t$-algebra.  Then for any $\e>0$ and any positive integer $\sigma$
  there exists a constant $C(\cB,\e,\delta,\sigma)$ such that for every fiber
  $(\Sp\cB)_\alpha$ as in Theorem~{\rm \ref{thm:main-uniform}},
$$
\#(\Sp\gdp\cB 1)_\alpha^\tran(F_{\alpha,0},H) \le C(\cB,\e,\delta,\sigma_{\alpha})\cdot H^\e, \qquad F_{\alpha,0}=
\begin{cases}
  \QQ & (F_\alpha\supset\QQ_p), \\
  \FF_q(t) & (F_\alpha\supset\FF_q\dpl t\dpr).
\end{cases}
$$
\end{thm}

Note that here we avoid the extra $q_\alpha^n$-term, making the
result truly uniform over all $\QQ_p$- and $\FF_q\dpl t\dpr$-points.
The proof of Theorem~\ref{thm:main-uniform-oc} is given in
Section~\ref{proof-main-uniform-oc}. Theorem~\ref{thm:main-uniform}
is obtained from Theorem~\ref{thm:main-uniform-oc} by covering
$\Sp\cB$ by $q_\alpha^n$ polydiscs of radius
$|t|$, and applying Theorem~\ref{thm:main-uniform-oc} to each of
them. This is carried out in Section~\ref{proof-main-uniform}.

Theorem~\ref{thm:main-uniform} is a special case of
Theorem~\ref{thm-main-family}, which establishes a similar result
where the rigid analytic space $\Sp\cB$ is also allowed to vary in a
rigid analytic family, and the constant $C(\cB,\e,\delta,\sigma)$ is uniformly
bounded over the family. Similar statements have been obtained in the
original Pila-Wilkie setting \cite{pw} as well as in
\cite{ccl:pw,cfl:uniform-pw}.

\subsection{Comparison with earlier work}
The work of \cite{ccl:pw} on Pila-Wilkie counting for
$\QQ_p$-subanalytic sets is roughly analogous to the proof in the
classical case.  The key difficulty is to find a suitable replacement
for the the reparametrization lemma (proved by Yomdin
\cite{yomdin:lemma} and Gromov \cite{gromov:gy} in the algebraic
setting, and extended to the general o-minimal structure by Pila and
Wilkie \cite{pw}).  This is accomplished in \cite{ccl:pw} through a
systematic study of Lipschitz continuous cell decompositions. In
\cite{cfl:uniform-pw} this is extended to $\ZZ\dbl t\dbr$.  The model
theoretic machinery then allows one to specialize uniformly to every
non-archimedean local field of sufficiently high characteristic.

Our proofs follow a different approach to the Pila-Wilkie theorem
introduced in \cite{bn:analytic-pw}. This approach avoids the use of
the reparamterization theorem and replaces it by an argument in the
spirit of Weierstrass preparation: instead of covering an analytic set
by smooth charts as in the reparametrization lemma, one covers it by
\emph{Weierstrass polydiscs} where analytic functions can be
Weierstrass prepared.  It turns out, perhaps unsurprisingly, that
Noether normalizaton provides a very direct analog of the Weierstrass
polydisc construction in the rigid analytic setting.  Our main
technical result in this direction,
Theorem~\ref{thm-stratifiednormalization}, shows that Noether
normalization can be performed uniformly in families -- giving a
suitable replacement of the reparametrization lemma.  The main
advantage of this approach is that Noether normalization is much
easier to carry out in positive characteristics than the
reparameterization approach of \cite{ccl:pw,cfl:uniform-pw}, and it is
this feature that allows us to carry out our proofs over 
$\FF_q\dpl t\dpr$ and indeed uniformly over all characteristics.

\begin{rem}
  In this paper we restrict attention to the setting of affinoid
  spaces, which is less general than the subanalytic spaces considered
  in \cite{ccl:pw,cfl:uniform-pw}.  In the archimedean setting, the
  approach of \cite{bn:analytic-pw} was also used with relatively
  little effort to recover the subanalytic case from the analytic
  case.  It seems possible that an analogous approach would yield a
  ``subanalytic'' version of our result in the rigid analytic setting
  as well.  However, to our knowledge the corresponding notion of
  $\rR t$-subanalytic sets in rigid geometry has not yet been
  developed in the literature.  We do remark that the approach of
  Martin \cite{martin:subanalytic} seems quite suitable for our
  purposes, if carried out in the more general setting of
  $\rR t$-analytic spaces.
\end{rem}

\subsection{Toward polylogarithmic counting theorems}
The Weierstrass polydisc construction developed in
\cite{bn:analytic-pw} has played the central role in many further
developments around the Pila-Wilkie theorem, concerning questions of
effectivity \cite{me:noetherian-pw} as well as in the direction of the
Wilkie conjecture, i.e. the improvement of the asymptotic $O(H^\e)$ to
a polylogarithmic $(\log H)^\kappa$
\cite{bn:rest-wilkie,me:qfol-geometry}. In the non-archimedean context,
the same idea was used in \cite{bcn:Ct-wilkie} to prove a
polylogarithmic counting result for germs of analytic varieties
defined by Pfaffian or Noetherian functions. Namely, the number of
$\CC(t)$-rational curves in such a germ grows polynomially as a
function of the degree.

Since our approach here establishes a very direct analog of the
Weierstrass polydisc construction in rigid analytic geometry, it seems
likely that it could lead to similar improvements in the
rigid analytic setting. We intend to pursue these applications in
forthcoming work. To illustrate the potential of this approach we
prove the following polylogarithmic interpolation result, which can be
seen as a first step toward polylogarithmic counting.

\begin{thm}\label{thm:main-uniform-pl}
  Let $\cB$ be an affinoid $\rR t$-algebra. Then there exists a constant
$$
C=\begin{cases} C(\cB,\delta,[F:\QQ_p]/r_{\alpha})&(F_\alpha\supset\QQ_p), \\
  C(\cB,\delta) & (F_\alpha\supset\FF_q\dpl t\dpr),
  \end{cases}
$$
such that for every fiber $(\Sp\cB)_\alpha$ as
  in Theorem~{\rm \ref{thm:main-uniform}}, the set
  $(\Sp\gdp\cB\delta)_\alpha^\tran(F_{\alpha,0},H)$ is contained in an
  algebraic hypersurface of degree $C(\log_{q_{\alpha}} H)^{d}$.
\end{thm}

Similar results have been obtained in \cite{ccl:pw} in the $\QQ_p$
setting and in \cite{cfl:uniform-pw} uniformly for sufficiently large
primes. However, experience from the classical $\RR$ context
\cite{bn:rest-wilkie} and the $\CC\dpl t\dpr$ context
\cite{bcn:Ct-wilkie} shows that the approach using Weierstrass
polydiscs makes it easier to extend such results into a full-fledged
polylogarithmic counting theorem.

\subsection{Organization of this paper}

This paper is organized as follows. In Section~\ref{sec-rigidgeom} we
develop some bacgkround material on rigid analytic geometry over
$\rR{t}$, where $\mathcal{O}$ is an integer ring of a number field. In
Section~\ref{sec-noether-normalization} we develop a uniform version
of the Noether normalization theorem in this context, which forms our
rigid analytic analog of the ``Weierstrass polydisc'' construction of
\cite{bn:analytic-pw}. In Section~\ref{sec-interpolation} we prove a
result on interpolation of rational points in rigid spaces using
algebraic hypersurfaces, assuming that the space is Noether
normalized. Finally in Section~\ref{sec-counting-proofs} we prove the
point counting theorems.

\subsection{Acknowledgements}

We would like to thank Raf Clckers for several corrections and many
important comments on the first version of this manuscript.

\subsection{Conventions}\label{sub-conv}
\begin{itemize}
\item We denote by $\NN$ the set of all non-negative integers.
\item For a local ring $R$ we denote by $\mm_R$ its maximal ideal.
\item For $v=a/b\in\QQ$ a reduced fraction, we denote
  $H(v)=\max\{|a|,|b|\}$. For $v=a(t)/b(t)\in\FF_q\dpl t\dpr$ a
  reduced fraction we denote $H(v)=\max\{q^{\deg a},q^{\deg b}\}$.
\end{itemize}

\section{Rigid geometry}\label{sec-rigidgeom}
\subsection{Admissible algebras}\label{sub-admissiblealgebras}
Let $F_0$ be a number field, fixed once for all, and
$\mathcal{O}=\mathcal{O}_{F_0}$ the integer ring of $F_0$.  We
consider the formal power series ring $\fR{t}$ with the $t$-adic
topology.  We denote by $\fR{t}\dl \bfmath{x}\dr$ (where
$\bfmath{x}=(x_1,\ldots,x_n)$) the {\em restricted power series ring
  over $\fR{t}$}, i.e., the $t$-adic completion of the polynomial ring
$\fR{t}[\bfmath{x}]$.

\begin{dfn}\label{dfn-admissiblealgebras}
  (1) A {\em topologically of finite type $\fR{t}$-algebra} is an
  $\fR{t}$-algebra $A$ that is isomorphic to an $\fR{t}$-algebra of
  the form $\fR{t}\dl \bfmath{x}\dr/\mathfrak{a}$, where
  $\mathfrak{a}\subset\fR{t}\dl \bfmath{x}\dr$ is an ideal.  By a {\em
    morphism} of topologically of finite type $\fR{t}$-algebras we
  mean an $\fR{t}$-algebra homomorphism.

  (2) A topologically of finite type $\fR{t}$-algebra is said to be
  {\em admissible} if it is $t$-torsion free.
\end{dfn}

Note that topologically of finite type $\fR{t}$-algebras are
Noetherian and $t$-adically complete, and every morphism between them
is adic.  We denote by $\Ad_{\fR{t}}$ the category of admissible
$\fR{t}$-algebras and $\fR{t}$-algebra homomorphisms.

The notion of topologically of finite type $\fR{t}$-algebras gives
rise to the notion of finite type formal schemes over $\fR{t}$.  Note
that, if $X$ is a finite type formal scheme over $\fR{t}$, then $X_0$
(the closed fiber by $t=0$) is a finite type $\mathcal{O}$-scheme,
hence is Jacobson.

\subsubsection{Restricted power series ring over $A$}
For any topologically of finite type $\fR{t}$-algebra $A$, we denote
by $A\dl\bfmath{x}\dr$ the $t$-adic completion of the polynomial ring
$A[\bfmath{x}]$, or what amounts to the same,
$$
A\dl\bfmath{x}\dr=\bigg\{\sum_{\nu\in\NN^n}a_{\nu}\bfmath{x}^{\nu}\in A\dbl\bfmath{x}\dbr\,\bigg|\,{\small \begin{minipage}{18em}for
    any $m\in\ZZ_{>0}$ there exists $M\in\ZZ_{>0}$ such that
    $|\nu|\geqq M\ \Rightarrow\ a_{\nu}\in t^mA$\end{minipage}}\bigg\}.
$$

\begin{lem}\label{lem-t-torsion}
  Let $A$ be a topologically of finite type $\fR{t}$-algebra, and
  $N\subset A$ its $t$-torsion part, i.e., the ideal consisting of the
  $t$-torsion elements.  Then the $t$-torsion part of
  $A\dl \bfmath{x}\dr$ is the ideal
$$
\bigg\{\sum_{\nu\in\NN^n}a_{\nu}\bfmath{x}^{\nu}\,\bigg|\,a_{\nu}\in N\ \textrm{for
  any}\ \nu\in\NN^n\bigg\}.\eqno{(\ast)}
$$
Thus we have
$$
A\dl\bfmath{x}\dr/(\textrm{$t$-torsion})\cong(A/(\textrm{$t$-torsion}))\dl \bfmath{x}\dr.
$$
In particular, if $A$ is admissible, then so is $A\dl\bfmath{x}\dr$.
\end{lem}

\begin{proof}
  Let $M$ be the $t$-torsion part of $A\dl \bfmath{x}\dr$.  It is
  clear that $M$ is contained in the ideal $(\ast)$.  Since $A$ is
  Noetherian, there exists sufficiently large $m>0$ such that
  $t^mN=0$.  Hence, if $F\in A\dl \bfmath{x}\dr$ belongs to the ideal
  $(\ast)$, then we have $t^mF=0$, and hence $F\in M$.
\end{proof}

\subsection{Classical points}\label{sub-classicalpoints}
We consider rigid spaces of finite type over
$\mathcal{S}=(\Spf\fR{t})^{\rig}$.  A finite type affinoid over
$\mathcal{S}$, for example, is a rigid space isomorphic to a rigid
space of the form $(\Spf A)^{\rig}$ given by a topologically of finite
type $\fR{t}$-algebra $A$.

Let $\mathcal{X}$ be a rigid space of finite type over $\mathcal{S}$.
Simply by a {\em point} of $\mathcal{X}$, we usually mean a point of
the associated Zariski-Riemann space $\ZR{\mathcal{X}}$.  A {\em
  classical point} of $\mathcal{X}$ is a retro-compact point-like
rigid subspace of $\mathcal{X}$ (\cite[{\bf II}, 8.2.8]{FK2}).  Note
that any classical point of $\mathcal{X}$ is closed (\cite[{\bf II},
8.2.9]{FK2}).  As usual, the set of all classical points of
$\mathcal{X}$ is denoted by $\ZR{\mathcal{X}}^{\cl}$.  If
$\mathcal{X}=(\Spf A)^{\rig}$, where $A$ is an admissible
$\fR{t}$-algebra, then classical points are in canonical one-to-one
correspondence with closed points of the Noetherian scheme
$\Spec\mathcal{A}$, where $\mathcal{A}=A[1/t]$ (\cite[{\bf II},
8.2.11]{FK2}).

\begin{rem}\label{rem-classicalrigpoints}
  The classical points in the situation of classical rigid geometry
  are often referred to as {\em rig-points}; cf.\
  \cite[\S8.3]{SBosch}.
\end{rem}

\begin{lem}\label{lem-classicalpoints1}
  Any classical point of a rigid space of finite type over
  $\mathcal{S}$ is isomorphic to $(\Spf V)^{\rig}$, where $V$ is a
  complete discrete valuation ring $V$ with finite residue field.
\end{lem}

\begin{proof}
  By \cite[{\bf II}, 8.2.6]{FK2}, the classical points are of the form
  $(\Spf V)^{\rig}$ by a complete discrete valuation ring $V$.  Since
  $V$ is topologically of finite type over $\fR{t}$, $V/tV$ is of
  finite type over $\mathcal{O}$.  As the residue field of $V$ is of
  finite type over $\mathcal{O}$, it is a finite field.
\end{proof}

\begin{cor}\label{cor-classicalpoints1}
  If $A$ is a topologically of finite type $\fR{t}$-algebra, and $\mm$
  is a maximal ideal of $\mathcal{A}=A[1/t]$, then the residue field
  $\mathcal{A}/\mm$ is a non-archimedean local field, i.e., a complete
  discrete valuation field with finite residue field.
\end{cor}

\begin{exa}\label{exa-classicalpoint1}
  Let us consider the case $\OO=\ZZ$, and let $p$ be a prime number.

  (1) A classical point
  $(\Spf\ZZ_p)^{\rig}\hookrightarrow(\Spf\fZ{t})^{\rig}$ is given by
  the surjective homomorphism $\fZ{t}\rightarrow\ZZ_p$ that maps $t$
  to $p$.

  (2) A classical point
  $(\Spf\FF_p\dbl t\dbr)^{\rig}\hookrightarrow(\Spf\fZ{t})^{\rig}$ is
  given by the canonical homomorphism
  $\fZ{t}\rightarrow\FF_p\dbl t\dbr\cong\fZ{t}/p\fZ{t}$.
\end{exa}

\begin{exa}\label{exa-classicalpoint2}
  Consider the homomorphism
$$
\fZ{t}\longrightarrow R=\fZ{t}/(p^3-t^2).
$$
The normalization $\til{R}$ of $R$ is a ramified quadratic extension
of $\ZZ_p$.  In fact, $\til{R}$ can be obtained from the strict
transform of the admissible blow-up $X\rightarrow\Spf\fZ{t}$ along the
admissible ideal $J=(p,t)$; i.e.,
$$
X=\Spf\fZ{t}\dl t/p\dr\cup\Spf\fZ{t}\dl p/t\dr\longrightarrow\Spf\fZ{t},
$$
and the strict transform on $\Spf R$ is
$\Spf\til{R}\rightarrow\Spf R$, where
$\til{R}\cong\fZ{t}\dl t/p\dr/(p-(t/p)^2)$.  In particular, the closed
immersion $\Spf\til{R}\hookrightarrow X$ gives rise to a classical
point of $(\Spf\fZ{t})^{\rig}$.
\end{exa}

\begin{lem}\label{lem-classicalpoints1-3}
  Let $\phi\colon\fR{t}\rightarrow V$, where $V$ is as in Lemma {\rm
    \ref{lem-classicalpoints1}}, be an adic morphism arised from a
  classical point of a rigid space of finite type over $\mathcal{S}$,
  and set $\qq=\phi^{-1}(\mm_V)$ and $\pp=\qq\cap\OO$.
\begin{itemize}
\item[{\rm (a)}] $\qq$ $($resp.\ $\pp)$ is a maximal ideal of $\fR{t}$
  $($resp.\ $\OO)$, and $\qq=\pp\fR{t}+(t)$;
\item[{\rm (b)}] $V$ is finite over $\fR{t}$.
\end{itemize}
\end{lem}

\begin{proof}
  Since the residue field $k=V/\mm_V$ is finite, $\qq$ (resp.\ $\pp$)
  is a maximal ideal of $\fR{t}$ $($resp.\ $\OO)$.  Since
  $\phi(t)\in\mm_V$, we have $\pp\fR{t}+(t)\subset\qq$.  But, since
  $\fR{t}/(\pp\fR{t}+(t))\cong\OO/\pp$, $\pp\fR{t}+(t)$ is maximal,
  and hence we have $\pp\fR{t}+(t)=\qq$.  
  Since $V/tV$ is a finite ring, $V/tV$ is finite over $\fR{t}/t\fR{t}$.
  Then by \cite[8.4]{Matsu}, $V$ is finite over  $\fR{t}$.
\end{proof}

\begin{lem}[Functoriality of classical points]\label{lem-functoriality}
  Let $\mathcal{X}\rightarrow\mathcal{Y}$ be a morphism of rigid
  spaces of finite type over $\mathcal{S}$, and
  $(\Spf V)^{\rig}\hookrightarrow\mathcal{X}$ a classical point.  Then
  there exists uniquely a commutative diagram
$$
\xymatrix{(\Spf V)^{\rig}\,\ar@{^{(}->}[r]\ar[d]&\mathcal{X}\ar[d]\\ (\Spf W)^{\rig}\,\ar@{^{(}->}[r]&\mathcal{Y},}
$$
where $W\rightarrow V$ is finite and
$(\Spf W)^{\rig}\hookrightarrow\mathcal{Y}$ is a classical point.
Thus we have the map
$\ZR{\mathcal{X}}^{\cl}\rightarrow\ZR{\mathcal{Y}}^{\cl}$ between the
set of classical points.
\end{lem}

\begin{proof}
  We may assume that $\mathcal{X}=(\Spf A)^{\rig}$ and
  $\mathcal{Y}=(\Spf B)^{\rig}$, and that
  $\mathcal{X}\rightarrow\mathcal{Y}$ is induced from a morphism
  $B\rightarrow A$ of admissible $\fR{t}$-algebras.  Then any
  classical point of $\mathcal{X}$ corresponds to an adic morphism
  $A\rightarrow V$ to a complete discrete valuation ring $V$ with
  finite residue field.  Since classical points in our situation are
  closed, we may assume that $A\rightarrow V$ is surjective.

  Let $\phi\colon\fR{t}\rightarrow V$ be given by the composition
  $\fR{t}\rightarrow B\rightarrow A\rightarrow V$, and set
  $\qq=\phi^{-1}(\mm_V)$ and $\pp=\qq\cap\OO$.  
  Note that, since $\sum_{\nu\geq 0}a_{\nu}t^{\nu}\in\fR{t}$ does not belong to $\qq$ 
  if and only if $a_0\not\in\pp$, we have 
  $\widehat{\fR{t}_{\qq}}=\OO_{\pp}\dbl t\dbr$. 
  Let $R$ be the image of
  $\widehat{\fR{t}_{\qq}}$ in $V$, which is a
  $t$-adically complete local subring of $V$, and $\til{R}$ the
  normalization of $R$ in $V$, which is a $t$-adically complete
  discrete valuation subring of $V$.  Let $J\subset R$ be the
  admissible ideal such that $\til{R}$ is the admissible blow-up along
  $J$.

  Consider the base change
  $B_R=B\widehat{\otimes}_{\fR{t}}R\rightarrow A_R=A\widehat{\otimes}_{\fR{t}}R$
  and the strict transform $B_{\til{R}}\rightarrow A_{\til{R}}$ by
  $R\rightarrow\til{R}$.  Since $A_{\til{R}}\rightarrow V$ is
  surjective, it defines a classical point of
  $\mathcal{X}_{\til{R}}=(\Spf A_{\til{R}})^{\rig}$.  By the functoriality
  of classical points over valuation rings of height $1$ (\cite[{\bf
    II}, 8.2.14]{FK2}), one has a commutative diagram of the form
$$
\xymatrix{\Spf V\,\ar@{^{(}->}[r]\ar[d]&\Spf A_{\til{R}}\ar@/^1.5pc/[dd]\\
  \Spf W\, \ar@{^{(}->}[r]^(.6){i}&Z\ar[d]_{\pi}\\ &\Spf B_{\til{R}},}
$$
where $W$ is a complete discrete valuation ring with finite residue
field, $W\rightarrow V$ is finite, $i$ is a closed immersion, and
$\pi$ is an admissible blow-up.

Let $\pi'\colon Z\rightarrow\Spf B_R$ be the composition
$Z\stackrel{\pi}{\rightarrow}\Spf B_{\til{R}}\rightarrow\Spf B_R$ of
two admissible blow-ups, which is again an admissible blow-up
(\cite[{\bf II}, 1.1.10]{FK2}), and $J'\subset B_R$ the corresponding
admissible ideal.  Let
$J_1\subset B\widehat{\otimes}_{\fR{t}}\fR{t}_{\qq}$ be the pull-back
of $J'$ by the surjection
$B\widehat{\otimes}_{\fR{t}}\fR{t}_{\qq}\rightarrow B_R$.  Then $J_1$
is an admissible ideal of $B\widehat{\otimes}_{\fR{t}}\fR{t}_{\qq}$,
and the corresponding admissible blow-up
$\pi_1\colon Z_1\rightarrow\Spf B\widehat{\otimes}_{\fR{t}}\fR{t}_{\qq}$
gives rise to $\pi'$ by passage to the strict transform.  Note that
$Z\hookrightarrow Z_1$ is a closed immersion.  Suppose $t^N\in J_1$.
Then $J_1$ corresponds to a finitely generated ideal $\ovl{J}_1$ of
$B\widehat{\otimes}_{\fR{t}}\fR{t}_{\qq}/(t^N)=B\widehat{\otimes}_{\fR{t}}\OO_{\pp}\dbl t\dbr/(t^N)$,
which extends to a finitely generated ideal $\ovl{J}_2$ of
$B\otimes_{\fR{t}}\OO[1/a][t]/(t^N)$ for some $a\in\OO\setminus\pp$.
Then the pull-back
$J_2\subset B\widehat{\otimes}_{\fR{t}}\OO[1/a]\dbl t\dbr$ of
$\ovl{J}_2$ gives rise to the admissible blow-up
$\pi_2\colon Z_2\rightarrow\Spf B\widehat{\otimes}_{\fR{t}}\OO[1/a]\dbl t\dbr$
that extends $\pi_1$.  Now, due to \cite[{\bf II}, 1.1.9]{FK2}, we can
further extend $\pi_2$ to an admissible blow-up
$\pi_3\colon Y\rightarrow\Spf B$.  Thus we have a locally closed
immersion $\Spf W\rightarrow Y$, which defines a classical point of
$\mathcal{Y}=Y^{\rig}$ (cf.\ \cite[{\bf II}, 8.2.9 (2)]{FK2}), and
fits in with the following commutative diagram
$$
\xymatrix{\Spf V\,\ar@{^{(}->}[r]\ar[d]&\Spf A_R\ar@/^1.5pc/[dd]\,\ar@{^{(}->}[r]&\Spf A\ar@/^1.5pc/[dd]\\
  \Spf W\, \ar@{^{(}->}[r]^(.6){i}&Z\ar[d]_{\pi}\,\ar@{^{(}->}[r]|(.28)\hole&Y\ar[d]\\ &\Spf B_R\,\ar@{^{(}->}[r]&\Spf B.}
$$
The uniqueness follows from \cite[{\bf II}, 8.2.11]{FK2}.
\end{proof}

\begin{cor}\label{cor-functoriality}
  Let $A\rightarrow B$ be a morphism of topologically of finite type
  $\fR{t}$-algebras.  Then $\Spec B[1/t]\rightarrow\Spec A[1/t]$ maps
  closed points to closed points.
\end{cor}

\subsection{Affinoid algebras}\label{sub-affinoidalgebras}
Set $\rR{t}=\fR{t}[1/t]$.
\begin{dfn}\label{dfn-affinoidalgebras}
  An {\em affinoid $\rR{t}$-algebra} is an $\rR{t}$-algebra that is
  isomorphic to an $\rR{t}$-algebra of the form $A[1/t]$ by a
  topologically of finite type $\fR{t}$-algebra $A$.  By a {\em
    morphism} of affinoid $\rR{t}$-algebras we mean an
  $\rR{t}$-algebra homomorphism.
\end{dfn}

\begin{dfn}\label{dfn-formalmodel}
  (1) A {\em formal model} of an affinoid $\rR{t}$-algebra
  $\mathcal{A}$ is a topologically of finite type $\fR{t}$-algebra $A$
  with an isomorphism $A[1/t]\stackrel{\sim}{\rightarrow}\mathcal{A}$
  of $\rR{t}$-algebras.

  (2) A formal model is said to be {\em admissible} if it is given by
  an admissible $\fR{t}$-algebra.
\end{dfn}

If $A$ is a formal model of $\mathcal{A}$, then
$A/(\textrm{$t$-torsion})$ gives an admissible formal model.  In
particular, any affinoid $\rR{t}$-algebra $\mathcal{A}$ admits an
admissible formal model.  Note that a formal model $A$ of
$\mathcal{A}$ is admissible if and only if the map
$A\rightarrow\mathcal{A}$ is injective.

We denote by $\Af_{\rR{t}}$ the category of affinoid $\rR{t}$-algebras
and $\rR{t}$-algebra homomorphisms.  There is a functor
$$
(\cdot)[1/t]\colon \Ad_{\fR{t}}\longrightarrow\Af_{\rR{t}},\qquad A\longmapsto A[1/t].
$$

\begin{prop}\label{prop-jacobson1}
  Any affinoid $\rR{t}$-algebra is Jacobson.
\end{prop}

\begin{proof}
  It suffices to show that, for any affinoid $\rR{t}$-domain
  $\mathcal{A}$, the intersection of all maximal ideals of
  $\mathcal{A}$ is zero.  Let $\pp$ be the kernel of
  $\OO\rightarrow\mathcal{A}$, which is a prime ideal of $\OO$.  If
  $\pp\neq(0)$, then $\pp$ is a maximal ideal of $\OO$, and
  $\mathcal{A}$ is an affinoid algebra over the complete discrete
  valuation field $(\OO/\pp)\dpl t\dpr$.  Hence $\mathcal{A}$ is known
  to be Jacobson.  Suppose $\pp=(0)$.  For any non-zero $f\in A$,
  consider the non-empty open subset $D(f)=\Spec\mathcal{A}_f$ of
  $\Spec\mathcal{A}$.  We need to show that $D(f)$ contains a closed
  point of $\Spec\mathcal{A}$.  Since the image of
  $\phi\colon\Spec\mathcal{A}\rightarrow\Spec\OO$ is dense and $\OO$ is Jacobson, there
  exists a closed point $x$ of $\Spec\OO$ such that
  $\phi^{-1}(x)\cap D(f)\neq\emptyset$.  The fiber $\phi^{-1}(x)$ is
  given by $\Spec\mathcal{A}/\qq\mathcal{A}$, where $\qq$ is the
  maximal ideal of $\OO$ corresponding to $x$, and
  $\mathcal{A}/\qq\mathcal{A}$ is an affinoid algebra over the
  complete discrete valuation field $(\OO/\qq)\dpl t\dpr$.  Since
  $\mathcal{A}/\qq\mathcal{A}$ is known to be Jacobson, there exists a
  closed point of the closed subset
  $\phi^{-1}(x)\subset\Spec\mathcal{A}$ in $\phi^{-1}(x)\cap D(f)$.
\end{proof}

\subsubsection{Tate algebra}
Let $A$ be an admissible $\fR{t}$-algebra, and $\mathcal{A}=A[1/t]$ the associated affinoid $\rR{t}$-algebra.
We set 
\begin{equation*}
  \begin{split}
    \mathcal{A}\dl\bfmath{x}\dr&:=A\dl\bfmath{x}\dr[1/t]\\
    &=\bigg\{\sum_{\nu\in\NN^n}a_{\nu}\bfmath{x}^{\nu}\in\mathcal{A}\dbl\bfmath{x}\dbr\,\bigg|\,{\small \begin{minipage}{18em}for
        any $m\in\ZZ_{>0}$ there exists $M\in\ZZ_{>0}$ such that
        $|\nu|\geqq M\ \Rightarrow\ a_{\nu}\in t^mA$\end{minipage}}\bigg\},
  \end{split}
\end{equation*}
which is an affinoid $\rR{t}$-algebra, called the {\em Tate algebra}
over $\mathcal{A}$.

\subsection{Norms on affinoid algebras}\label{sub-norms}
The ring $\rR{t}$ is equipped with a non-archimedean norm
$$
|\cdot|\colon\rR{t}\longrightarrow\RR_{\geq 0}
$$
defined as follows: for $a(t)=\sum_{k\geq n}a_kt^k$ with $a_n\neq 0$,
we set $|a(t)|=e^{-n}$, where $e>1$ is a real number fixed once for
all, and $|0|=0$.  Note that $\rR{t}$ is complete with respect to this
norm.

For any maximal ideal $\mm$ of $\rR{t}$, the residue field
$K=\rR{t}/\mm$ is a non-archimedean local field (Lemma
\ref{cor-classicalpoints1}), and has the unique norm $|\cdot|_K$ such
that $|\ovl{t}|_K=e^{-1}$, where $\ovl{t}$ is the image of $t$ in $K$.
Note that, for any $a(t)\in\rR{t}$, we have
$$
|\ovl{a}(\ovl{t})|_K\leq|a(t)|,
$$
where $\ovl{a}(\ovl{t})$ is the image of $a(t)$ in $K$.

\subsubsection{Gauss norm}\label{sub-gaussnorm}
The Tate algebra $\rR{t}\dl \bfmath{x}\dr$ (where
$\bfmath{x}=(x_1,\ldots,x_n)$) is equipped with the {\em Gauss norm}
$\|\cdot\|$, which is defined, as usual, as follows:
$$
\bigg\|\sum_{\nu}a_{\nu}\bfmath{x}^{\nu}\bigg\|=\max_{\nu}|a_{\nu}|.
$$
As usual, one can show that the Gauss norm is multiplicative, i.e.,
$\|FG\|=\|F\|\cdot\|G\|$, and that the Tate algebra
$\rR{t}\dl \bfmath{x}\dr$ is complete with respect to the Gauss norm.

For any maximal ideal $\mm$ of $\rR{t}$, the Tate algebra
$\rR{t}\dl \bfmath{x}\dr$ is specialized to the usual Tate algebra
$K\dl \bfmath{x}\dr$ over $K=\rR{t}/\mm$, and if $\|\cdot\|_K$ denotes
the usual Gauss norm on $K\dl \bfmath{x}\dr$ relative to the norm
$|\cdot|_K$ on $K$ as above, we have
$$
\|\ovl{F}(\bfmath{x})\|_K\leq\|F(\bfmath{x})\|
$$
for any $F(\bfmath{x})\in\rR{t}\dl \bfmath{x}\dr$, where
$\ovl{F}(\bfmath{x})$ denotes the image of $F(\bfmath{x})$ in
$K\dl \bfmath{x}\dr$.

\subsubsection{Residue norm}\label{subsub-residuenorm}
For an affinoid $\rR{t}$ -algebra $\mathcal{A}$, with a presentation
$\mathcal{A}\cong\rR{t}\dl\bfmath{x}\dr/\mathfrak{a}$, one has the
{\em residue norm} $\|\cdot\|_{\mathcal{A}}$ defined by
$$
\|f\|_{\mathcal{A}}=\inf\{\|F\|\mid\textrm{$F\in \rR{t}\dl\bfmath{x}\dr$
  and $(F\, \mathrm{mod}\, \mathfrak{a})=f$}\}
$$
for $f\in\mathcal{A}$.

For any maximal ideal $\mm$ of $\rR{t}$, the presentation
$\mathcal{A}\cong\rR{t}\dl\bfmath{x}\dr/\mathfrak{a}$ is specialized
to a presentation
$\mathcal{A}_K\cong K\dl\bfmath{x}\dr/\mathfrak{a}K\dl\bfmath{x}\dr$
of the usual affinoid algebra $\mathcal{A}_K$ over $K=\rR{t}/\mm$, and
thus gives rise to a residue norm $\|\cdot\|_{\mathcal{A}_K}$ on
$\mathcal{A}_K$ (induced from the Gauss norm on $K\dl\bfmath{x}\dr$
defined as above).  We have
$$
\|\ovl{f}\|_{\mathcal{A}_K}\leq\|f\|_{\mathcal{A}}\eqno{(\ast)}
$$
for any $f\in\mathcal{A}$, where $\ovl{f}$ is the image of $f$ in
$\mathcal{A}$.

\subsection{Topology on affinoid algebras}\label{sub-topologyaffinoid}
Any residue norm on an affinoid algebra $\mathcal{A}$ defines a
complete topology on $\mathcal{A}$.  We now show that this topology
does not depend on the choice of the residue norm.

\begin{lem}\label{lem-affinoidcontinuous}
  Let $\mathcal{A}$ and $\mathcal{B}$ be affinoid $\rR{t}$-algebras
  considered with arbitrary residue norms.  Then any $\rR{t}$-algebra
  homomorphism $\varphi\colon\mathcal{A}\rightarrow\mathcal{B}$ is
  continuous.
\end{lem}

\begin{proof}
  Due to closed graph theorem for metrically complete topological
  groups \cite{Pettis}, we only need to show the following: if a
  sequence $\{f_n\}$ in $\mathcal{A}$ converges to $0$ and
  $\{\varphi(f_n)\}$ converges to an element $g$ in $\mathcal{B}$,
  then $g=0$.  Take any maximal ideal $\mm\subset\mathcal{B}$, and let
  $\nn\subset\rR{t}$ be the image of $\mm$ by
  $\Spec\mathcal{B}\rightarrow\Spec\rR{t}$, which is a maximal ideal
  of $\rR{t}$ due to Corollary \ref{cor-functoriality}.  Set
  $K=\rR{t}/\nn$ (cf.\ Corollary \ref{cor-classicalpoints1}).  For any
  $l\geq 1$, $\varphi_l\colon\mathcal{A}\rightarrow\mathcal{B}/\mm^l$
  is continuous, since
  $\mathcal{A}/\ker(\varphi_l)\hookrightarrow\mathcal{B}/\mm^l$ is a
  mapping between finite dimensional $K$-linear spaces, and the
  topology of $\mathcal{A}/\ker(\varphi_l)$ is the induced one from
  that of $\mathcal{A}$.  Hence we have $g\in\mm^l$ for any $l\geq 1$.
  Then by Krull's theorem (\cite[Chap.\ III, \S3.2, Cor]{Bourb1}), the
  ideal $\mathrm{Ann}(g)=\{x\in\mathcal{B}\mid xg=0\}$ is not
  contained in any maximal ideal of $\mathcal{B}$.  Hence
  $1\in\mathrm{Ann}(g)$, i.e., $g=0$.
\end{proof}

\begin{cor}\label{cor-affinoidcontinuous}
  The residue norm on an affinoid algebra $\mathcal{A}$ gives a
  well-defined topology on $\mathcal{A}$; i.e., the induced topology
  does not depend on the choice of a presentation
  $\mathcal{A}\cong\rR{t}\dl\bfmath{x}\dr/\mathfrak{a}$.
\end{cor}

Hence in the sequel, we will always consider affinoid
$\rR{t}$-algebras with the {\em canonical} topology as above.  Any
$\rR{t}$-algebra homomorphism between affinoid algebras is continuous
with respect to the canonical topology.

\begin{lem}\label{lem-affinoidtopologyadic}
  Let $\mathcal{A}$ be an affinoid $\rR{t}$-algebra, and $A$ an
  admissible formal model of $\mathcal{A}$.  Then the restriction to
  $A$ of the canonical topology on $\mathcal{A}$ coincides with the
  $t$-adic topology.
\end{lem}

\begin{proof}
  The assertion is clear if $\mathcal{A}=\rR{t}\dl\bfmath{x}\dr$ and
  $A=\fR{t}\dl\bfmath{x}\dr$, i.e., the topology by the Gauss norm on
  $\fR{t}\dl\bfmath{x}\dr$ coincides with the $t$-adic topology.  In
  general, take a presentation
  $A\cong\fR{t}\dl\bfmath{x}\dr/\mathfrak{a}$, and the induced
  presentation
  $\mathcal{A}\cong\rR{t}\dl\bfmath{x}\dr/\mathfrak{a}\rR{t}\dl\bfmath{x}\dr$.
  Then the topology on $A$, which is the restriction of the topology
  by the induced residue norm, is the quotient topology of the
  $t$-adic topology, hence is the $t$-adic topology.
\end{proof}

\begin{cor}\label{cor-formalmodelmorphism}
  Let $\varphi\colon\mathcal{A}\rightarrow\mathcal{B}$ be a morphism
  of affinoid $\rR{t}$-algebras, and $A$ an admissible formal model of
  $\mathcal{A}$.  Then there exists an admissible formal model $B$ of
  $\mathcal{B}$ and a morphism $\phi\colon A\rightarrow B$ such that
  $\varphi=\phi[1/t]$.
\end{cor}

\begin{proof}
  Take an admissible formal model $B$ of $\mathcal{B}$, and let $B'$
  be the image of the continuous morphism
  $A\widehat{\otimes}_{\fR{t}}B\rightarrow\mathcal{B}$.  Since $B'$
  contains $B$, $B'$ is a formal model of $\mathcal{B}$,
  which admits $A\rightarrow B'$, as desired.
\end{proof}

\begin{cor}\label{cor-formalmodelmorphism1}
  Let $\mathcal{A}\rightarrow\mathcal{B}$ be a morphism of affinoid
  $\rR{t}$-algebras.  Then
  $\Spec\mathcal{B}\rightarrow\Spec\mathcal{A}$ maps closed points to
  closed points.
\end{cor}

\begin{proof}
  Immediate from Corollary \ref{cor-formalmodelmorphism} and Corollary
  \ref{cor-functoriality}.
\end{proof}

\subsection{Reduction map}\label{sub-reductionmap}
Let $\mathcal{A}$ be an affinoid $\rR{t}$-algebra, and
$A\subset\mathcal{A}$ an admissible formal model of $\mathcal{A}$.
Set $A_0=A/tA$.  Then we have the so-called {\em reduction map}
$$
\red_A\colon\Spm\mathcal{A}\longrightarrow\Spm A_0
$$
from the maximal spectrum of $\cA$ to the maximal spectrum of $A_0$, defined as follows.  Any
closed point $x\in\Spec\mathcal{A}$ corresponds to a classical point
$(\Spf V)^{\rig}\rightarrow(\Spf A)^{\rig}$, where $V$ is a complete
discrete valuation ring with finite residue field, whence a finite
morphism $\Spf V\rightarrow\Spf A$ of $t$-adic formal schemes.  The
last morphism induces a finite morphism $\Spec k\rightarrow\Spec A_0$,
where $k$ is the residue field of $V$, and hence a closed point of
$\Spec A_0$, which we define to be $\red_A(x)$.

The reduction map can be defined in a little more general situation.
Let $X'\rightarrow X=\Spf A$ be an admissible blow-up.  Then the
finite adic map $\Spf V\rightarrow\Spf A$ has a unique lift
$\Spf V\rightarrow X'$, which gives rise to a finite morphism
$\Spec k\rightarrow X'_0$, where $X'_0$ is the closed subscheme of
$X'$ defined by $t=0$, which is a finite type $\OO$-scheme.  We thus
have a reduction map
$$
\red_{X'}\colon\Spm\mathcal{A}\longrightarrow (X'_0)^{\cl},
$$
to the set of closed points of $X'_0$.

\begin{lem}\label{lem-reductionmap}
  The reduction map $\red_{X'}$ is surjective.
\end{lem}

\begin{proof}
  First, note that there exists a map $(X'_0)^{\cl}\rightarrow\Spm\OO$, since $X'_0$ is of finite type
  over $\OO$, which is Jacobson. Hence $(X'_0)^{\cl}$ is disjoint union of the fibers
  of this map.  The fiber over $x\in\Spm\OO$ of the
  reduction map $\red_{X'}$ is the reduction map
  $\Spm\mathcal{A}/\pp_x\mathcal{A}\rightarrow (X'_{0,x})^{\cl}$ in
  classical rigid geometry, where $\pp_x\subset\OO$ is the maximal
  ideal corresponding to $x$; note that $\mathcal{A}/\pp_x\mathcal{A}$
  is an affinoid algebra over the complete discrete valuation field
  $(\OO/\pp_x)\dpl t\dpr$, and the reduction map is considered with
  respect to the formal model $A/\pp_x A$ of
  $\mathcal{A}/\pp_x\mathcal{A}$.  As the reduction map in classical
  situation is known to be surjective (e.g.\ \cite[\S8.3, Prop.\
  8]{SBosch}), the assertion follows.
\end{proof}

\subsection{Power-bounded elements}\label{sub-spectralseminorms}
Let $\mathcal{A}$ be an affinoid $\rR{t}$-algebra.  For any closed
point $x\in\Spm\mathcal{A}$, the residue field $K_x$ at $x$ is a
non-archimedean local field (Corollary \ref{cor-classicalpoints1}),
and has the unique non-archimedean norm $|\cdot|_x$ defined as in
\S\ref{sub-norms}.  One has a map
$$
\mathcal{A}\longrightarrow\RR_{\geq 0},\quad f\longmapsto |f(x)|:=|f|_x,
$$
and, for any $f\in\mathcal{A}$, the so-called {\em spectral seminorm}
$$
|f|_{\sp}=\sup\{|f(x)|\mid x\in\Spm\mathcal{A}\},
$$
which is power-multiplicative, i.e., $|f^n|_{\sp}=(|f|_{\sp})^n$ for
any $n\geq 0$.

\begin{lem}\label{lem-spectralbdd}
  The spectral seminorm is bounded by any residue norm, i.e.,
  $|f|_{\sp}\leq\|f\|_{\mathcal{A}}$ for any $f\in\mathcal{A}$, where
  $\|\cdot\|_{\mathcal{A}}$ is a residue seminorm relative to an
  arbitrary presentation
  $\mathcal{A}\cong\rR{t}\dl\bfmath{x}\dr/\mathfrak{a}$.
\end{lem}

\begin{proof}
  For any closed point $x\in\Spm\mathcal{A}$, let $y\in\Spm\rR{t}$ be
  its image (cf.\ Corollary \ref{cor-formalmodelmorphism1}), and $K$
  the residue field at $y$.  Then $|f(x)|=|\ovl{f}(x)|$, where
  $\ovl{f}$ is the image of $f$ in $\mathcal{A}_K$.  It is known in
  the classical rigid geometry that
  $|\ovl{f}(x)|\leq\|\ovl{f}\|_{\mathcal{A}_K}$, and by $(\ast)$ in
  \S\ref{subsub-residuenorm}, we have $|f(x)|\leq\|f\|_{\mathcal{A}}$.
\end{proof}

\begin{prop}\label{prop-powerbounded}
  Let $\mathcal{A}$ be an affinoid $\rR{t}$-algebra.  Then the
  following conditions for $f\in\mathcal{A}$ are equivalent.
  \begin{itemize}
  \item[{\rm (a)}] $f$ is {\em power-bounded}, i.e., the set $\{f^n\}$
    is bounded by a residue norm $|\cdot|_{\mathcal{A}}$ on
    $\mathcal{A}$ $($note that the boundedness does not depend on the
    choice of $|\cdot|_{\mathcal{A}})$;
  \item[{\rm (b)}] $|f|_{\sp}\leq 1$;
  \item[{\rm (c)}] $f$ is integral over any admissible formal model
    $A$ of $\mathcal{A}$;
  \item[{\rm (d)}] there exists an admissible formal model $A$ of
    $\mathcal{A}$ that contains $f$.
  \end{itemize}
\end{prop}

\begin{proof}
  The spirit of the proof is just the same as that of the
  corresponding theorem in the classical rigid geometry.  However, we
  include the proof here for the reader's convenience.

  Let us first show (a) $\Rightarrow$ (b).  Suppose that $f$ is
  power-bounded, and $|f|_{\sp}>1$.  Since $|\cdot|_{\sp}$ is
  power-multiplicative, $\{|f^n|_{\sp}\}$ is not bounded, and hence
  $\{|f^n|_{\mathcal{A}}\}$ for any residue norm
  $|\cdot|_{\mathcal{A}}$ is not bounded, either.

  Second we show (c) $\Rightarrow$ (d).  Suppose $f$ is integral over
  an admissible formal model $A$ of $\mathcal{A}$.  Then $A[f]$ is
  finite over $A$, hence is topologically of finite type, which gives
  an admissible formal model of $\mathcal{A}$.

  Next we show (d) $\Rightarrow$ (a).  Take an admissible formal model
  $A$ such that $f\in A$, and a presentation
  $A\cong\fR{t}\dl\bfmath{x}\dr/\mathfrak{a}$.  Then we have
  $\mathcal{A}=\rR{t}\dl\bfmath{x}\dr/\mathfrak{a}\rR{t}\dl\bfmath{x}\dr$,
  and one has the residue norm $|\cdot|_{\mathcal{A}}$ on
  $\mathcal{A}$ with respect to this presentation.  Since
  $|g|_{\mathcal{A}}\leq 1$ whenever $g\in A$, we have (a).

  Finally, let us show (b) $\Rightarrow$ (c).  Let $A$ be an
  admissible formal model of $\mathcal{A}$.  Take sufficiently large
  $N\geq 0$ such that $g=t^Nf\in A$, and consider the admissible ideal
  $J=(t^N,g)$ of $A$.  Let $X'\rightarrow X=\Spf A$ be the admissible
  blow-up along $J$.  We have $X=U^+\cup U^-$, where
  \begin{equation*}
    \begin{split}
      U^+&=\Spf A\dl g/t^n\dr/(\textrm{$t$-torsion}),\\
      U^-&=\Spf A\dl t^n/g\dr/(\textrm{$g$-torsion}).
    \end{split}
  \end{equation*}
  We are going to show $X'=U^+$.  What to show is that the closed
  subscheme $Z_0$ of $Z=X'\setminus U^+$ defined by $t=0$ is empty.
  First note that $Z_0$ is a finite type scheme over $\OO$.  By our
  assumption (b), for any $x\in\Spm\mathcal{A}$, the image $f(x)$ of
  $f$ in the residue field $K_x$ at $x$ belongs to the valuation ring
  $V_x$.  Now the map $\Spf V_x\rightarrow\Spf A$ lifts to
  $\Spf V_x\rightarrow X'$ whose image lies in $U^+$, since $g=t^nf$.
  This means that the image of the reduction map
  $\Spm\mathcal{A}\rightarrow X'_0$ lies in $U^+_0$.  Since the
  reduction map $\Spm\mathcal{A}\rightarrow (X'_0)^{\cl}$ is
  surjective by Lemma \ref{lem-reductionmap}, we deduce that $Z_0$ has
  no closed point of $X'_0$.  Since $X'_0$ is of finite type over
  $\OO$, and hence is Jacobson, this means $Z_0=\emptyset$.

  Now, since $\pi\colon X'\rightarrow X=\Spf A$ is proper and affine,
  it is finite.  In particular, $\Gamma(X,\pi_{\ast}\OO_{X'})=A[f]$ is
  finite over $A$, as desired.
\end{proof}

\begin{cor}\label{cor-powerbounded1}
  Let $\mathcal{A}$ be an affinoid $\rR{t}$-algebra, and
  $f_1,\ldots,f_n\in\mathcal{A}$ power-bounded elements.  Then there
  exists an admissible formal model $A$ of $\mathcal{A}$ that contains
  $f_1,\ldots,f_n$.
\end{cor}

\begin{cor}\label{cor-powerbounded2}
  Let $\varphi\colon\mathcal{A}\rightarrow\mathcal{B}$ be a morphism
  of affinoid $\rR{t}$-algebras, and $g_1,\ldots,g_n\in\mathcal{B}$
  power-bounded elements of $\mathcal{B}$.  Then $\varphi$ lifts
  uniquely to a morphism
  $\mathcal{A}\dl x_1,\ldots,x_n\dr\rightarrow\mathcal{B}$ such that
  $x_i\mapsto g_i$ for $i=1,\ldots,n$.
\end{cor}

\begin{cor}\label{cor-powerbounded3}
  Let $\varphi\colon\mathcal{A}\rightarrow\mathcal{B}$ be a morphism
  of affinoid $\rR{t}$-algebras, and $A$ $($resp.\ $B)$ an admissible
  formal model of $\mathcal{A}$ $($resp.\ $\mathcal{B})$.  Then there
  exists an admissible formal model $B'$ of $\mathcal{B}$ such that
  \begin{itemize}
  \item[(a)] $B'$ admits a homomorphism $\phi\colon A\rightarrow B'$
    of admissible $\fR{t}$-algebras such that $\phi[1/t]=\varphi$;
  \item[(b)] $B\subset B'$, and $B'$ is finite over $B$;
  \item[(c)] $\Spf B'\rightarrow\Spf B$ is an admissible blow-up.
  \end{itemize}
\end{cor}

\begin{proof}
  Take a presentation $A=\fR{t}\dl x_1,\ldots,x_n\dr/\mathfrak{a}$,
  and set $f_i=(x_i\ \mathrm{mod}\ \mathfrak{a})$ for $i=1,\ldots,n$.
  Set $g_i=\varphi(f_i)$ for $i=1,\ldots,n$.  Then, since
  $|f_i|_{\sp}\leq 1$, we have $|g_i|_{\sp}\leq 1$, and thus
  $B'=B[g_1,\ldots,g_n]$ is an admissible formal model of
  $\mathcal{B}$, which is finite over $B$.

  We need to show that $\Spf B'\rightarrow\Spf B$ is an admissible
  blow-up.  To this end, take sufficiently large $N\geq 0$ such that
  $t^N g_i\in B$ for any $i=1,\ldots,n$, and consider the admissible
  blow-up $X''\rightarrow\Spf B$ along the admissible ideal
  $J=(t^N,t^Ng_1,\ldots,t^Ng_n)$.  Then $X''$ is covered by $\Spf B_i$
  for $i=0,1,\ldots,n$, where
$$
B_0=B\bigg\langle\!\!\!\bigg\langle\frac{t^Ng_1}{t^N},\ldots,\frac{t^Ng_n}{t^N}\bigg\rangle\!\!\!\bigg\rangle/(\textrm{$t$-torsion}),
$$
which is nothing but $B'$, and for $i=1,\ldots,n$,
$$
B_i=B\bigg\langle\!\!\!\bigg\langle\frac{t^N}{t^Ng_i},\frac{t^Ng_1}{t^Ng_i},\ldots,\frac{t^Ng_n}{t^Ng_i}\bigg\rangle\!\!\!\bigg\rangle/(\textrm{$t^Ng_i$-torsion}),
$$
which is isomorphic to the complete localization $B'_{\{g_i\}}$ of
$B'$.  Hence we have $X''=\Spf B'$, and thus we have shown that
$\Spf B'\rightarrow\Spf B$ is an admissible blow-up, as desired.
\end{proof}

\begin{cor}\label{cor-powerbounded5}
  Let $\mathcal{A}$ be an affinoid $\rR{t}$-algebra, and $A$ and $A'$
  admissible formal models of $\mathcal{A}$ such that $A\subset A'$.
  Then $A'$ is finite over $A$.  Moreover, $\Spf A'\rightarrow\Spf A$
  is an admissible blow-up.
\end{cor}

\begin{proof}
  Since $A'$ is topologically of finite type over $A$, one can write
  $A'=A\dl f_1,\ldots,f_n\dr$ for some $f_1,\ldots,f_n\in A'$.  By
  Proposition \ref{prop-powerbounded}, each $f_i$ is integral over
  $A$, and hence $A'=A[f_1,\ldots,f_n]$, which is finite over $A$.
  The other part can be shown similarly to the proof of Corollary
  \ref{cor-powerbounded3}.
\end{proof}

\begin{cor}\label{cor-powerbounded4}
  Let $\mathcal{A}$ be an affinoid $\rR{t}$-algebra, and $A$ and $A'$
  admissible formal models of $\mathcal{A}$.  Then there exists an
  admissible formal model $A''$ of $\mathcal{A}$ such that
  \begin{itemize}
  \item[(a)] $A''$ contains both $A$ and $A'$, and $A''$ is finite
    over $A$ and $A'$;
  \item[(b)] $\Spf A''\rightarrow\Spf A$ and
    $\Spf A''\rightarrow\Spf A'$ are admissible blow-ups.
  \end{itemize}
\end{cor}

\begin{proof}
  Take a presentation $A=\fR{t}\dl x_1,\ldots,x_n\dr/\mathfrak{a}$
  (resp.\ $A'=\fR{t}\dl x'_1,\ldots,x'_m\dr/\mathfrak{a}'$), and set
  $f_i=(x_i\ \mathrm{mod}\ \mathfrak{a})$ for $i=1,\ldots,n$ (resp.\
  $f'_j=(x'_j\ \mathrm{mod}\ \mathfrak{a}')$ for $j=1,\ldots,m$).
  Then $f_i$ and $f'_j$ are power-bounded elements, and set
  $A''=A[f'_1,\ldots,f'_m]=A'[f_1,\ldots,f_n]$.  Then this $A''$
  satisfies the condition (a).  The condition (b) follows from Corollary \ref{cor-powerbounded5}.
\end{proof}

\begin{rem}\label{rem-rigidification}
  What we have seen in Corollary \ref{cor-powerbounded3} and Corollary
  \ref{cor-powerbounded4} sums up to say in the language of \cite[{\bf
    II}, \S A.4.(g)]{FK2} that affinoid $\rR{t}$-algebras, like
  classical affinoid algebras, have the so-called {\em canonical
    rigidification}.
\end{rem}

\subsection{Rigid analytic geometry over $\rR{t}$}\label{sub-Sp}
By Corollary \ref{cor-powerbounded4}, for any $\rR{t}$-affinoid
algebra $\mathcal{A}$, the affinoid space $(\Spf A)^{\rig}$ does not
depend on the choice of the admissible formal model $A$ of
$\mathcal{A}$.  Hence we can write, as in the classical rigid geometry,
$$
\Sp\mathcal{A}:=(\Spf A)^{\rig}.
$$

\begin{rem}\label{rem-harmless}
  Note that this definition of $\Sp\mathcal{A}$ differs (harmlessly)
  from the one in the classical situation, where affinoid spaces in
  the classical setting are defined as the analytic spaces supported
  on the set of all maximal ideals (i.e., classical points) of
  $\mathcal{A}$.
\end{rem}

By Corollary \ref{cor-powerbounded3}, any $\rR{t}$-algebra
homomorphism $\mathcal{A}\rightarrow\mathcal{B}$ between affinoid
$\rR{t}$-algebras induces a well-defined morphism
$\Sp\mathcal{B}\rightarrow\Sp\mathcal{A}$ of rigid spaces over
$\rR{t}$.

\subsubsection{$K$-valued points}\label{subsub-valuedpoints}
Let $\fR{t}\rightarrow K$ be a finite $\rR{t}$-algebra, where $K$ is a non-archimedean local field. 
Then, for any affinoid space $\Sp\mathcal{A}$, one can consider a {\em $K$-valued point} $\alpha\colon\Sp K\rightarrow\Sp\mathcal{A}$ as a morphism between rigid spaces over $\rR{t}$.
By Corollary \ref{cor-formalmodelmorphism1}, $\alpha$ factors as 
$$
\xymatrix@C-3ex{\Sp K\ar[rr]^{\alpha}\ar[dr]&&\Sp\mathcal{A}\rlap{,}\\ &\Sp L\ar@{^{(}->}[ur]_{\beta}}
$$
where $K/L$ is a finite extension of non-archimedean local fields, and $\beta$ is a classical point.

\subsubsection{Specialization to the classical rigid geometry}\label{subsub-specialization}
Consider a $K$-valued point $\Sp K\hookrightarrow\Sp\fR{t}$, where $K$ is
a non-archimedean local field.  Then any affinoid space
$\Sp\mathcal{A}$ over $\rR{t}$ can be specialized to the affinoid
space $\Sp\mathcal{A}_K$ over $K$ in the sense of classical rigid
geometry, where $\mathcal{A}_K=\mathcal{A}\otimes_{\rR{t}}K$.

More generally, for any finite type rigid space $\mathcal{X}$ over
$\rR{t}$, one has the base change $\mathcal{X}_K$, which is a finite
type rigid space over $K$.

\subsection{Extension by roots of $t$}\label{sub-rootoft}
For a positive integer $N$, one has the finite extension
$\fR{t}\rightarrow\fR{t}[t^{1/N}]=\fR{t,s}/(s^N-t)$.  Since the latter
ring is $t$-adically complete, it coincides with its completion
$\fR{t^{1/N}}$.  Note that $\fR{t^{1/N}}=\fR{t}[t^{1/N}]$ is
faithfully flat over $\fR{t}$, and that
$\fR{t^{1/N}}[1/t]=\rR{t^{1/N}}$.  Note also that, if $A$ is
$t$-adically complete, then
$A\otimes_{\fR{t}}\fR{t^{1/N}}=A\widehat{\otimes}_{\fR{t}}\fR{t^{1/N}}$,
since the left-hand side is already $t$-adically complete.

Let $\mathcal{A}$ be an affinoid integral domain over $\rR{t}$.
Consider the simple field extension $\Frac(\mathcal{A})(t^{1/N})$ of
$\Frac(\mathcal{A})$ generated by an $N$th root of $t$.  We denote by
$$
\mathcal{A}[t^{1/N}]
$$
the subring of $\Frac(\mathcal{A})(t^{1/N})$ generated by
$\mathcal{A}$ and $t^{1/N}$.  This is an affinoid
$\rR{t^{1/N}}$-algebra, since there exists a surjection
$\mathcal{A}\otimes_{\rR{t}}\rR{t^{1/N}}\rightarrow\mathcal{A}[t^{1/N}]$
and
$\mathcal{A}\otimes_{\rR{t}}\rR{t^{1/N}}=\mathcal{A}\widehat{\otimes}_{\rR{t}}\rR{t^{1/N}}$
is an affinoid $\rR{t^{1/N}}$-algebra.

If $A$ is an admissible formal model of $\mathcal{A}$, one can
similarly define
$$
A[t^{1/N}]
$$
to be the subring of $\Frac(A)(t^{1/N})$ generated by $A$ and
$t^{1/N}$, which is the surjective image of the canonical map
$A\otimes_{\fR{t}}\fR{t^{1/N}}=A\widehat{\otimes}_{\fR{t}}\fR{t^{1/N}}\rightarrow A[t^{1/N}]$.
Clearly, $A[t^{1/N}]$ is an admissible formal model of
$\mathcal{A}[t^{1/N}]$.

\subsection{Affinoid algebra with a convergence condition}
In the sequel, by the {\em denominator} of a rational number $\dt$, we
mean the smallest positive integer $d(\dt)$ such that
$d(\dt)\cdot\dt\in\ZZ$.  Similarly, the denominator of a vector
$\bfmath{\dt}=(\dt_1,\ldots,\dt_n)\in\QQ^n$ of rational numbers is the
smallest positive integer $d(\bfmath{\dt})$ such that
$d(\bfmath{\dt})\cdot\bfmath{\dt}\in\ZZ^n$.  We write
$$
t^{\bfmath{\dt}}=(t^{\dt_1},\ldots,t^{\dt_n})\quad\textrm{and}\quad t^{\bfmath{\dt}}\bfmath{x}=(t^{\dt_1}x_1,\ldots,t^{\dt_n}x_n),
$$
and for $\nu=(\nu_1,\ldots,\nu_n)\in\NN^n$,
$$
\bfmath{\dt}\cdot\nu=\dt_1\nu_1+\cdots+\dt_n\nu_n
$$

\subsubsection{Affinoid algebra with a convergence condition}
For $\bfmath{\dt}=(\dt_1,\ldots,\dt_n)\in(\QQ_{\geq 0})^n$ and an
affinoid algebra $\mathcal{A}=A[1/t]$, where $A$ is an admissible
$\fR{t}$-algebra, we set
\begin{equation*}
  \begin{split}
    &\mathcal{A}\dl\bfmath{x};\bfmath{\dt}\dr=\bigg\{\sum_{\nu\in\NN^n}a_{\nu}\bfmath{x}^{\nu}\in\mathcal{A}\dbl\bfmath{x}\dbr\,\bigg|\,{\small
      \begin{minipage}{11.5em}$|t|^{-\bfmath{\dt}\cdot\nu}|a_{\nu}|\rightarrow 0$
        as $|\nu|\rightarrow\infty$\end{minipage}}\bigg\}
  \end{split}
\end{equation*}
(not depending on the choice of a residue norm $|\cdot|$ on
$\mathcal{A}$), and
\begin{equation*}
  \begin{split}
    &A\dl\bfmath{x};\bfmath{\dt}\dr=\bigg\{\sum_{\nu\in\NN^n}a_{\nu}\bfmath{x}^{\nu}\in\mathcal{A}\dl\bfmath{x};\bfmath{\dt}\dr\,\bigg|\,{\small |a_{\nu}|\leqq |t|^{\bfmath{\dt}\cdot\nu}}\bigg\},
  \end{split}
\end{equation*}
which is an $A$-algebra.  Note that the condition
$|a_{\nu}|\leqq |t|^{\bfmath{\dt}\cdot\nu}$ is equivalent to
$a^{d(\bfmath{\dt})}_{\nu}\in t^{d(\bfmath{\dt})(\bfmath{\dt}\cdot\nu)}A$,
hence is independent of choice of the residue norm $|\cdot|$.

If $\dt_i=m_i/d_i$ is the irreducible fraction with $d_i>0$ for
$i=1,\ldots,n$, then
$$
A\dl\bfmath{x};\bfmath{\dt}\dr\cong A\dl\bfmath{x},\bfmath{y}\dr/\mathfrak{a},
$$
where $\bfmath{y}=(y_1,\ldots,y_n)$ and $\mathfrak{a}$ is the ideal
generated by $t^{m_i}x^{d_i}_i-y^{d_i}_i$ for $i=1,\ldots,n$.  In
particular, $A\dl\bfmath{x};\bfmath{\dt}\dr$ is an admissible
$\fR{t}$-algebra, and hence
$\mathcal{A}\dl\bfmath{x};\bfmath{\dt}\dr=A\dl\bfmath{x};\bfmath{\dt}\dr[1/t]$
is an affinoid $\rR{t}$-algebra.  Moreover, it is easy to see that, if
$\bfmath{x}'=(x'_1,\ldots,x'_m)$ and
$\bfmath{\dt}'=(\dt'_1,\ldots,\dt'_m)\in(\QQ_{\geq 0})^m$, then
\begin{equation*}
  \begin{split}
    \mathcal{A}\dl\bfmath{x},\bfmath{x}';(\bfmath{\dt},\bfmath{\dt}')\dr&\cong\mathcal{A}\dl\bfmath{x};\bfmath{\dt}\dr\widehat{\otimes}_{\mathcal{A}}\mathcal{A}\dl\bfmath{x}';\bfmath{\dt}'\dr,\\
    A\dl\bfmath{x},\bfmath{x}';(\bfmath{\dt},\bfmath{\dt}')\dr&\cong A\dl\bfmath{x};\bfmath{\dt}\dr\widehat{\otimes}_AA\dl\bfmath{x}';\bfmath{\dt}'\dr,
  \end{split}
\end{equation*}
where
$(\bfmath{\dt},\bfmath{\dt}')=(\dt_1,\ldots,\dt_n,\dt'_1,\ldots,\dt'_m)$.

Note that $\mathcal{A}\dl\bfmath{x};\bfmath{\dt}\dr$ is an affinoid
algebra corresponding to the polydisk of polyradius
$|t|^{-\bfmath{\dt}}=(|t|^{-\dt_1},\ldots,|t|^{-\dt_n})$, i.e.,
\begin{equation*}
  \begin{split}
    A\dl\bfmath{x};\bfmath{\dt}\dr\otimes_{\fR{t}}\fR{t^{1/N}}&\cong(A\otimes_{\fR{t}}\fR{t^{1/N}})\dl t^{\bfmath{\dt}}\bfmath{x}\dr,\\
    \mathcal{A}\dl\bfmath{x};\bfmath{\dt}\dr\otimes_{\rR{t}}\rR{t^{1/N}}&\cong(\mathcal{A}\otimes_{\fR{t}}\fR{t^{1/N}})\dl t^{\bfmath{\dt}}\bfmath{x}\dr,
  \end{split}
\end{equation*}
and, if $\mathcal{A}$ is an integral domain,
\begin{equation*}
  \begin{split}
    A\dl\bfmath{x};\bfmath{\dt}\dr[t^{1/N}]\cong A[t^{1/N}]\dl t^{\bfmath{\dt}}\bfmath{x}\dr,\ \mathcal{A}\dl\bfmath{x};\bfmath{\dt}\dr[t^{1/N}]\cong \mathcal{A}[t^{1/N}]\dl t^{\bfmath{\dt}}\bfmath{x}\dr.
  \end{split}
\end{equation*}
where $N$ is a multiple of $d(\bfmath{\dt})$.

\begin{lem}\label{lem-openimmersion}
Let $\bfmath{\dt}=(\dt_1,\ldots.\dt_n),\bfmath{\dt}'=(\dt'_1,\ldots,\dt'_n)\in(\QQ_{>0})^n$, and suppose $\bfmath{\dt}'\leq\bfmath{\dt}$, i.e., $\dt'_i\leq\dt_i$ for $i=1,\ldots,n$.
Then $\mathcal{A}\dl\bfmath{x};\bfmath{\dt}\dr\rightarrow\mathcal{A}\dl\bfmath{x};\bfmath{\dt'}\dr$ is flat.
\end{lem}

\begin{proof}
We may assume that $n=1$, i.e., it suffices to show that $\mathcal{A}\dl x;\dt\dr\rightarrow\mathcal{A}\dl x;\dt'\dr$ for $\delta'\leq\delta$ is flat.
Since the extension $\fR{t}\rightarrow\fR{t^{1/N}}$ is faithfully flat, we may replace $\mathcal{A}$ by $\mathcal{A}\otimes_{\fR{t}}\fR{t^{1/N}}$, and thus we may assume that $\delta,\delta'\in\ZZ$.
Then $\Sp\mathcal{A}\dl x;\dt'\dr\rightarrow\Sp\mathcal{A}\dl x;\dt\dr$ is an open immersion, since $\Spf A\dl x;\dt'\dr\rightarrow\Spf A\dl x;\dt\dr$ is one the affine patches of the admissible blow-up of $\Spf A\dl x;\dt\dr$ along the admissible ideal $(t^{\dt-\dt'},x)$.
Then the desired flatness follows from \cite[{\bf II}, 6.6.1 (1)]{FK2}. 
\end{proof}

\subsubsection{Unit-polydisk part}\label{subsub-unitpolydiskpart}
If $\bfmath{\dt}=(\dt_1,\ldots,\dt_n),\bfmath{\dt}'=(\dt'_1,\ldots,\dt'_n)\in(\QQ_{>0})^n$ with $\bfmath{\dt}'\leq\bfmath{\dt}$, then
$\mathcal{A}\dl\bfmath{x};\bfmath{\dt}\dr$ is contained in
$\mathcal{A}\dl\bfmath{x};\bfmath{\dt'}\dr$.  In this situation, if $\mathcal{B}$ is
an affinoid algebra over $\mathcal{A}\dl\bfmath{x};\bfmath{\dt}\dr$,
then we define
$$
\gdp{\mathcal{B}}{\bfmath{\dt}'}=\mathcal{B}\widehat{\otimes}_{\mathcal{A}\dl\bfmath{x};\bfmath{\dt}\dr}\mathcal{A}\dl\bfmath{x};\bfmath{\dt}'\dr.
$$
Especially, $\gdp{\cB}{0}$ will be called the {\em unit-polydisk part} of $\mathcal{B}$, which we denote by $\udp{\cB}$.

\subsection{Irreducible
  decomposition}\label{sub-irreducibledecomposition}
\begin{lem}\label{lem-irreducible1}
  Let $\mathcal{X}$ be a rigid space of finite type over $\rR{t}$, and
  $\mathcal{N}_{\mathcal{X}}$ the ideal sheaf of $\OO_{\mathcal{X}}$
  consisting of local nilpotent sections.  Then
  $\mathcal{N}_{\mathcal{X}}$ is a coherent ideal of
  $\OO_{\mathcal{X}}$.
\end{lem}

\begin{proof}
  This can be shown similarly to \cite[{\bf II}, 8.2.9]{FK2}, using
  the fact that any topologically of finite type $\fR{t}$-algebra is
  excellent (due to Gabber; see \cite{KS}).
\end{proof}

With the above lemma, one finds that all the rest of \cite[{\bf II},
\S8.3.(b)]{FK2} are valid for rigid spaces of finite type over
$\rR{t}$.  In particular, an affinoid
$\Sp\mathcal{A}=(\Spf A)^{\rig}$, where $\mathcal{A}$ is an affinoid
$\rR{t}$-algebra and $A$ is an admissible formal model of
$\mathcal{A}$, is reduced (resp.\ irreducible) if and only if so is
the affine scheme $\Spec\mathcal{A}$ (cf.\ \cite[{\bf II},
8.3.5]{FK2}).  Thus one has the irreducible decomposition of
$\mathcal{X}=\Sp\mathcal{A}$; if
$$
\Spec\mathcal{A}=\bigcup^m_{i=1} X_i
$$
is the irreducible decomposition of the Noetherian scheme
$\Spec\mathcal{A}$, then
$$
\mathcal{X}=\bigcup^m_{i=1}\mathcal{X}_i,
$$
where $\mathcal{X}_i=\Sp\mathcal{A}/\mathfrak{q}_i$
($\cong(\Spf A/\mathfrak{q}_i\cap A)^{\rig}$) and $\mathfrak{q}_i$ is
the minimal prime ideal of $\mathcal{A}$ corresponding to $X_i$, gives
the irreducible decomposition of $\mathcal{X}$.

\subsection{Dimension}\label{sub-dimension}
We refer to \cite[{\bf II}, \S10]{FK2} for the general notion of dimension of rigid spaces.  
First of all, for a rigid space $\mathcal{X}$ and a point $x\in\ZR{\mathcal{X}}$, the dimension of $\mathcal{X}$ at $x$, denoted by $\dim_x(\mathcal{X})$, is the Krull dimension of the local ring $\OO_{\mathcal{X},x}$.  
The dimension $\dim(\mathcal{X})$ of $\mathcal{X}$ is then defined to be the supremum of $\dim_x(\mathcal{X})$ for all $x\in\ZR{\mathcal{X}}$.
Note that, if $\mathcal{X}$ is a rigid space of finite type over $\rR{t}$, then $\dim(\mathcal{X})$ is the supremum of $\dim_x(\mathcal{X})$ for all classical points $x$ of $\mathcal{X}$ (cf.\ \cite[{\bf II}, 10.1.9]{FK2}).  
Moreover, if $\mathcal{X}=\Sp\mathcal{A}$, where $\mathcal{A}$ is an affinoid $\rR{t}$-algebra, then $\dim_x(\mathcal{X})$ at any classical point $x$ coincides with $\dim_{s(x)}(\Spec\mathcal{A})$, where $s(x)$ is the closed point of $\Spec\mathcal{A}$ corresponding to $x$. In particular, we have
$$
{\textstyle \dim(\mathcal{X})=\dim(\mathcal{A})}.
$$

\begin{lem}[Fiber dimension theorem]\label{lem-fiberdim}
  Let $\varphi\colon\mathcal{X}\rightarrow\mathcal{Y}$ be a morphism
  between rigid spaces of finite type over $\rR{t}$,
  $y=\varphi(x)\in\ZR{\mathcal{Y}}^{\cl}$ a classical point, and
  $x\in\ZR{\mathcal{X}}^{\cl}$ a classical point over $y$, i.e.,
  $y=\varphi(x)$.  Then we have
$$
\dim_x(\mathcal{X})-\dim_y(\mathcal{Y})\leq\dim_x(\mathcal{X}_y).
$$
\end{lem}

\begin{proof}
  What we need to show is
$$
\dim\OO_{\mathcal{X},x}-\dim\OO_{\mathcal{Y},y}\leq\dim\OO_{\mathcal{X},x}/\mm_y\OO_{\mathcal{X},x},
$$
where $\mm_y$ is the maximal ideal of $\OO_{\mathcal{Y},y}$.  Since
the local rings $\OO_{\mathcal{X},x}$ and $\OO_{\mathcal{Y},y}$ are
Noetherian, this follows from \cite[15.1]{Matsu}.
\end{proof}

The following lemma will be used later.
\begin{lem}\label{lem-finiteinjection}
  Let
  $\phi\colon\mathcal{A}\dl y_1,\ldots,y_d;\dt_1,\ldots,\dt_d\dr\rightarrow\mathcal{B}$
  be a finite morphism of affinoid $\rR{t}$-algebras with nilpotent kernel. 
  Then, for any
  morphism $\mathcal{A}\rightarrow\mathcal{R}$ of affinoid
  $\rR{t}$-algebras, the induced map
$$
\phi_{\mathcal{R}}\colon\mathcal{R}\dl y_1,\ldots,y_d;\dt_1,\ldots,\dt_d\dr\longrightarrow\mathcal{B}_{\mathcal{R}}=\mathcal{B}\widehat{\otimes}_{\mathcal{A}}\mathcal{R}
$$
is finite and $\ker(\phi_{\mathcal{R}})$ is contained in the nilpotent
radical.
\end{lem}

\begin{proof}
First, since the extension of the form $\fR{t}\rightarrow\fR{t^{1/N}}$ is faithfully flat, one can first reduce to the case where $\dt_1,\ldots,\dt_d\in\ZZ$ by faithfully flat descent, and then to the case $\dt_1=\cdots=\dt_d=0$, by replacing each $y_i$ by $t^{\delta_i}y_i$ for $i=1,\ldots,d$.

Second, as the map in question is clearly finite, we only need to show that
  $\phi_{\mathcal{R}}$ consists of nilpotent elements.  Suppose not,
  and take a non-nilpotent $f\in\ker(\phi_{\mathcal{R}})$.  By
  Proposition \ref{prop-jacobson1}, there exists a classical point
  $x=\mm_x$ of $\Sp\mathcal{R}\dl y_1,\ldots,y_d\dr$ ($\mm_x$ is the
  corresponding maximal ideal of $\mathcal{R}\dl y_1,\ldots,y_d\dr$)
  such that $f\not\in\mm_x$.  Let $y=\mm_y$ be the image of $x$ in
  $\Sp\mathcal{R}$, which is a classical point of $\Sp\mathcal{R}$
  (Lemma \ref{lem-functoriality}).  Let $K=\mathcal{R}/\mm_y$, and
  consider the induced diagram
$$
K\longrightarrow K\dl y_1,\ldots,y_d\dr\stackrel{\phi_K}{\longrightarrow}\mathcal{B}_K,
$$
where $\phi_K$ is, by our construction, finite but not injective.  It
is clear that
$$
\dim\Sp\mathcal{B}_K=\dim K\dl y_1,\ldots,y_d\dr/\ker(\phi_K)<d.
$$
On the other hand, by Lemma \ref{lem-fiberdim},
$$
\dim\Sp\mathcal{B}_K\geq \dim_x\Sp\mathcal{B}-\dim_y\Sp\mathcal{A}=d,
$$
which is absurd.
\end{proof}

\subsection{Formal blow-up of affinoids}\label{sub-formalblowup}
Let $\mathcal{A}$ be an affinoid $\rR{t}$-algebra, and $A$ an
admissible formal model of $\mathcal{A}$.  Let $J=(f_0,\ldots,f_r)$ be
an ideal of $A$.  We want to discuss the {\em formal blow-up} of
$\Sp\mathcal{A}$ along the ideal $J$.  This includes the notion of
admissible blow-ups as the case where $J$ is an admissible ideal.

Consider the algebraic blow-up
$Y=\Proj\bigoplus_{k=0}^{\infty}J^k\rightarrow\Spec A$ along $J$.
Then
$$
Y=\bigcup^r_{i=0}\Spec A_i,\quad
A_i=A\bigg[\frac{f_0}{f_i},\ldots,\frac{f_r}{f_i}\bigg]/\textrm{$f_i$-torsion}.
$$
Taking formal completion,
$$ \widehat{Y}=\bigcup^r_{i=0}\Spf\widehat{A}_i,\quad\widehat{A}_i=A\bigg\langle\!\!\!\bigg\langle\frac{f_0}{f_i},\ldots,\frac{f_r}{f_i}\bigg\rangle\!\!\!\bigg\rangle/\textrm{$f_i$-torsion},
$$
and the associated rigid spaces, we have the desired formal blow-up
$\mathcal{Y}\rightarrow\Sp\mathcal{A}$, where
$$ \mathcal{Y}=\widehat{Y}^{\rig}=\bigcup^r_{i=0}\Sp\mathcal{A}_i,\quad\mathcal{A}_i=\widehat{A}_i[1/t]=\mathcal{A}\bigg\langle\!\!\!\bigg\langle\frac{f_0}{f_i},\ldots,\frac{f_r}{f_i}\bigg\rangle\!\!\!\bigg\rangle/\textrm{$f_i$-torsion}.
$$

\subsubsection{Essential center}
Let us say that the ideal $\mathcal{J}=J\mathcal{A}$ is the {\em
  essential center} of the formal blow-up
$\mathcal{Y}=\bigcup_{i=0}^r\Sp\mathcal{A}_i\rightarrow\Sp\mathcal{A}$.
Note that, if $J$ is an admissible ideal, then
$\mathcal{J}=\mathcal{A}$, i.e., the essential center is ``empty'' in
the affinoid space $\Sp\mathcal{A}$.

If $\Sp\mathcal{B}\rightarrow\Sp\mathcal{A}$ is a morphism of finite
type, take an admissible formal model $A\rightarrow B$ of
$\mathcal{A}\rightarrow\mathcal{B}$.  Then
$\mathcal{X}=\widehat{X}^{\rig}$, where $X$ is the formal blow-up of
$\Spec B$ along $JB$, and the map $X\rightarrow Y$ from the
universality of blow-up gives rise to the morphism
$\mathcal{X}\rightarrow\mathcal{Y}$, which we call the {\em strict
  transform} of the formal blow-up
$\mathcal{Y}\rightarrow\Sp\mathcal{A}$ by
$\Sp\mathcal{B}\rightarrow\Sp\mathcal{A}$.

\begin{lem}\label{lem-blow-up}
  Let $\Sp\mathcal{A}$ be an irreducible affinoid, and
  $\mathcal{Y}\rightarrow\Sp\mathcal{A}$ the formal blow-up along
  $J=(f_0,\ldots,f_r)\subset A$.  Suppose $J$ is not nilpotent.

  {\rm (1)} The morphism $\mathcal{Y}\rightarrow\Sp\mathcal{A}$ is
  surjective.  Moreover, any classical point
  $\Sp K\rightarrow\Sp\mathcal{A}$, where $K$ is a non-archimedean
  local field, lifts to a classical point
  $\Sp K\rightarrow\mathcal{Y}$ of $\mathcal{Y}$.

  {\rm (2)} If $\mathcal{A}$ is an integral domain, then so is
  $\mathcal{A}_i$ for $i=0,\ldots,r$.
\end{lem}

\begin{proof}
  (1) For any rigid point $\alpha\colon\Spf W\rightarrow\Spf A$, where
  $W$ is a $t$-adically complete valuation ring, the extension ideal
  $JW$ is invertible (since $J$ is finitely generated), and so there
  exists a factoring map $\Spec W\rightarrow Y$, which gives, by
  formal completion, a rigid point above $\alpha$.  The second
  assertion can be shown similarly.

  (2) If $\mathcal{A}$ is an integral domain, then so is $A$, since
  $A$ is $t$-torsion free.  Then $A_i$ ($i=0,\ldots,r$) are integral
  domains (\cite[$\mathbf{II}$, (8.1.4)]{EGA}).  Since $A_i$ are
  excellent, $\widehat{A}_i$ are integral domains, and so are
  $\mathcal{A}_i$ for $i=0,\ldots,r$.
\end{proof}

\section{Noether normalization for rigid spaces over $\rR{t}$}
\label{sec-noether-normalization}

\subsection{Diagram modeled on a rooted tree}
\subsubsection{Rooted tree}
In what follows, simply by a rooted tree, we mean a finite directed
rooted tree with the orientation away from the root.  As usual, we
will often regard a rooted tree $T$ as a category with the unique
minimal object.  The set of all vertices is denoted by $\Ve_T$, and
for any $v\in\Ve_T$, the set of all edges outgoing from $v$ is denoted
by $\Ed^+_T(v)$.  If $v\rightarrow u$ is a directed edge of $T$, we
say $v$ is a {\em parent} of $u$, and $u$ is a {\em child} of $v$.
Note that the set $\Ed^+_T(v)$ can be identified with the set of all
children of $v$.  We say a vertex $w$ is a {\em descendant} of $v$,
and $v$ is an {\em ancestor} of $w$, if there exists a directed path
from $v$ to $w$.  A maximal vertex, i.e., vertex with no children,
will be called a {\em leaf}.

A {\em basic tree} is a height $1$ rooted tree, i.e., a rooted tree
consisting of root and its children.  Any rooted tree can be
decomposed into the union of basic trees.

\subsubsection{Diagrams} 
We will consider functors of the form $T\rightarrow\Ad_{\fR{t}}$ or
$T\rightarrow\Af_{\rR{t}}$ from a rooted tree, which we shall simply
call {\em diagrams}.  A diagram of admissible algebras
$A_{\ast}\colon T\rightarrow\Ad_{\fR{t}}$ induces by composition with
the functor $(\cdot)[1/t]\colon\Ad_{\fR{t}}\rightarrow\Af_{\rR{t}}$
(cf.\ \S\ref{sub-affinoidalgebras}) the diagram
$T\rightarrow\Af_{\rR{t}}$ of affinoid algebras, which we denote by
$A_{\ast}[1/t]$.  {\em Basic diagrams} are diagrams modeled on a basic
tree.

Let $A_{\ast}$ be a diagram modeled on a rooted tree.  An {\em ideal}
$\mathfrak{a}_{\ast}$ of $A_{\ast}$ consists of ideals
$\mathfrak{a}_v\subset A_v$ for any vertex $v\in\Ve_T$ such that, for
any edge $v\rightarrow u$ of $T$,
$\mathfrak{a}_vA_u\subset \mathfrak{a}_u$.  In this situation, one can
construct the diagram $A_{\ast}/\mathfrak{a}_{\ast}$ given by
$A_v/\mathfrak{a}_v$ for $v\in\Ve_T$.

\begin{exa}\label{exa-basediagram}
  We will often use the diagrams of the following form, denoted by
$$
\fR{t^{1/N_{\ast}}}\qquad{\textrm{or}}\qquad\rR{t^{1/N_{\ast}}}\ (=\fR{t^{1/N_{\ast}}}[1/t]),
$$
defined as follows.  The diagram $\fR{t^{1/N_{\ast}}}$ modeled on a
rooted tree $T$ consists of $\fR{t^{1/N_v}}$ for vertices $v$, where
$N_v$ are positive integers, such that, for any edge $v\rightarrow u$
of $T$, $N_u$ is a multiple of $N_v$, and
$$
\fR{t^{1/N_v}}\longhookrightarrow\fR{t^{1/N_u}}
$$
is the unique morphism determined by
$t^{1/{N_v}}\mapsto (t^{1/{N_u}})^{N_u/N_v}$.
\end{exa}

\begin{dfn}\label{dfn-coveringdiagram}
  A {\em transformation diagram modeled on a rooted tree $T$} is a
  chain of morphisms
$$
\fR{t^{1/N_{\ast}}}\longrightarrow A_{\ast}\longrightarrow A_{\ast}\dl\bfmath{x}_{\ast};\bfmath{\dt}_{\ast}\dr\longrightarrow B_{\ast}=A_{\ast}\dl\bfmath{x}_{\ast};\bfmath{\dt}_{\ast}\dr/\mathfrak{a}_{\ast}\eqno{(\ast)}
$$
of diagrams of admissible $\fR{t}$-algebras modeled on $T$ satisfying
the following conditions:
\begin{itemize}
\item[(a)] for $v\in\Ve_T$, $\bfmath{x}_v=(x_{v,1},\ldots,x_{v,n})$
  and $\bfmath{\dt}_v=(\dt_{v,1},\ldots,\dt_{v,n})\in(\QQ_{>0})^n$,
  where the number $n$ of variables is fixed and does not depend on
  $v$;
\item[(b)] for any directed edge $v\rightarrow u$ of $T$, we have
  $\bfmath{\dt}_v\geq\bfmath{\dt}_u$, i.e., $\dt_{v,i}\geq\dt_{u,i}$
  for $i=1,\ldots,n$;
\item[(c)] for any directed edge $v\rightarrow u$,
  $\mathfrak{a}_v A_u\dl\bfmath{x}_u;\bfmath{\dt}_u\dr\subset\mathfrak{a}_u$.
\end{itemize}
\end{dfn}

In the following, we will write $\mathcal{A}_{\ast}=A_{\ast}[1/t]$ and
$\mathcal{B}_{\ast}=B_{\ast}[1/t]$.  The induced diagram
$$
\rR{t^{1/N_{\ast}}}\longrightarrow\mathcal{A}_{\ast}\longrightarrow\mathcal{A}_{\ast}\dl\bfmath{x}_{\ast};\bfmath{\dt}_{\ast}\dr\longrightarrow\mathcal{B}_{\ast}=\mathcal{A}_{\ast}\dl\bfmath{x}_{\ast};\bfmath{\dt}_{\ast}\dr/\mathfrak{a}_{\ast}\mathcal{A}_{\ast}\dl\bfmath{x}_{\ast};\bfmath{\dt}_{\ast}\dr\eqno{(\ast)[1/t]}
$$
will be called the {\em transformation diagram of affinoid
  $\rR{t}$-algebras modeled on $T$}.

\subsection{Rational diagrams}
Let us say that a coordinate $\bfmath{y}=(y_1,\ldots,y_n)$ is a {\em
  $\ZZ$-rational coordinate transform} of
$\bfmath{x}=(x_1,\ldots,x_n)$ if the coordinate transformation is
given by polynomials over $\ZZ$, i.e., $y_i=F_i(\bfmath{x})$ by
$F_i\in\ZZ[\bfmath{x}]$ and $x_j=G_j(\bfmath{y})$ by
$G_j\in\ZZ[\bfmath{y}]$ ($i,j=1,\ldots,n$).  Note that $\ZZ$-rational
coordinate transformation preserves the unit-polydisk, since any
polynomial mapping maps points in the unit-polydisk to a point in the
unit-polydisk.

\begin{dfn}\label{dfn-rationaldiagram1}
  (1) A {\em rational diagram modeled on a rooted tree $T$} is a
  transformation diagram of admissible $\fR{t}$-integral domains
  modeled on $T$ as in $(\ast)$ of Definition
  \ref{dfn-coveringdiagram} satisfying the following conditions:
  \begin{itemize}
  \item[(a)] every vertex $v$ is equipped with an ideal
    $J_v=(c_0,\ldots,c_{r_v})$ of $A_v$ with fixed generators, and the
    set $\Ed^+_T(v)$ of edges outgoing from $v$ consists of exactly
    $r_v+1$ elements $v\rightarrow u_i$ ($i=0,\ldots,r_v$) such that
    each $A_{u_i}$ is of the form $A_{u_i}=A_i[t^{1/d_i}]$, where
    $d_i$ is a positive integer, and
$$
A_i=A_v\bigg\langle\!\!\!\bigg\langle\frac{c_0}{c_i},\ldots,\frac{c_{r_v}}{c_i}\bigg\rangle\!\!\!\bigg\rangle/\textrm{$c_i$-torsion},
$$
i.e., $A_{u_i}$ is the $i$th affine patch of the formal blow-up (cf.\
Lemma \ref{lem-blow-up} (2)) of $A$ with respect to
$J_v=(c_0,\ldots,c_{r_v})$ with a base extension by a root of $t$;
$N_{u_i}$ is the least common multiple of $N_v$ and $d_i$;
\item[(b)] for any directed edge $v\rightarrow u$ of $T$,
  $\bfmath{x}_u=(x_{u,1},\ldots,x_{u,n})$ is a $\ZZ$-rational
  coordinate transform of $\bfmath{x}_v=(x_{v,1},\ldots,x_{v,n})$;
\item[(c)] each $B_{u_i}$ is of the form $B_{u_i}=B_i[t^{1/d_i}]$,
  where $d_i$ is as in (a), and $B_i$ is the strict transform
$$
B_i=\big(B_v\widehat{\otimes}_{A_v\dl\bfmath{x}_v;\bfmath{\dt}_v\dr}A_i\dl\bfmath{x}_{u_i};\bfmath{\dt}_{u_i}\dr\big)/\textrm{$c_i$-torsion}.
$$
\end{itemize}

(2) A {\em stratified rational diagram} modeled on a rooted tree $T$
is a transformation diagram of admissible $\fR{t}$-algebras modeled on
$T$ as in $(\ast)$ of Definition \ref{dfn-coveringdiagram} that
satisfies the following conditions:
\begin{itemize}
\item[(d)] for any $v\in\Ve_T$, the set $\Ed^+_T(v)$ of edges outgoing
  from $v$ is partitioned into two parts
$$
\Ed^+_T(v)=\Ed^{+,1}_T(v)\amalg\Ed^{+,2}_T(v),
$$
which we call {\em type $1$} and {\em type $2$} edges, respectively;

\item[(e)] if $\Ed^{+,1}_T(v)\neq\emptyset$, then $A_v$ and $B_v$ are
  integral domains, and the conditions (a) $\sim$ (c) in (1) with
  $\Ed^+_T(v)$ replaced by $\Ed^{+,1}_T(v)$ is satisfied; i.e.,
  $\Ed^{+,1}_T(v)$ consists of exactly $r_v+1$ elements
  $v\rightarrow u_i$ ($i=0,\ldots,r_v$), and $A_v\rightarrow A_{u_i}$
  are of the form as in (a) constructed from a formal blow-up of $A_v$
  along an ideal $J_v$;

\item[(f)] the type $2$ edges are of the form
  $\Ed^{+,2}_T(v)=\{v\rightarrow w_{ij}\mid i=1,\ldots,l,j=1,\ldots,m_i\}$
  such that
  \begin{itemize}
  \item[(f1)] $N_{w_{ij}}=N_v$, $\bfmath{x}_{w_{ij}}=\bfmath{x}_v$ and
    $\bfmath{\dt}_{w_{ij}}=\bfmath{\dt}_v$ for all $(i,j)$;
  \item[(f2)] for any $i=1,\ldots,l$, $A_{w_{ij}}$ are all equal to an
    $A_i$;
  \item[(f3)] $\mathcal{A}_i=A_i[1/t]$ and
    $\mathcal{B}_{w_{ij}}=B_{w_{ij}}[1/t]$ are given by the
    irreducible decompositions (\S\ref{sub-irreducibledecomposition})
$$
\Sp\mathcal{A}_v/J_v=\bigcup^l_{i=1}\Sp\mathcal{A}_i,\qquad \Sp\mathcal{B}_v\otimes_{\mathcal{A}_v}\mathcal{A}_i=\bigcup^{m_i}_{j=1}\Sp\mathcal{B}_{w_{ij}};
$$
if $\mathfrak{p}_i$ is the kernel of
$\mathcal{A}_v\rightarrow\mathcal{A}_i$ then
$A_i=A_v/\mathfrak{p}_i\cap A_v$; if $\mathfrak{q}_{ij}$ is the kernel
of
$\mathcal{A}_i\dl \bfmath{x}_v;\bfmath{\dt}_v\dr\rightarrow \mathcal{B}_v\otimes_{\mathcal{A}}\mathcal{A}_i\rightarrow\mathcal{B}_{w_{ij}}$,
then
$\mathfrak{a}_{w_{ij}}=\mathfrak{q}_{ij}\cap A_i\dl \bfmath{x}_v;\bfmath{\dt}_v\dr$,
and
$B_{w_{ij}}=A_i\dl \bfmath{x}_v;\bfmath{\dt}_v\dr/\mathfrak{a}_{w_{ij}}$.
\end{itemize}
\end{itemize}
In other words, the type $2$ edges comprise the basic diagram by
irreducible decomposition of the induced family over the essential
blow-up center.
\end{dfn}

\begin{rem}\label{rem-rationaldiagram0}
  Note that, if $A_{v}$ and $B_{v}$ is integral domains at a vertex
  $v$ of $T$, then by Lemma \ref{lem-blow-up} (2), $A_{w}$ and $B_{w}$
  for any descendants $w$ of $v$ are integral domains.
\end{rem}

\begin{rem}\label{rem-rationaldiagram1}
  If
  $\fR{t^{1/N_{\ast}}}\rightarrow A_{\ast}\rightarrow A_{\ast}\dl \bfmath{x}_{\ast};\bfmath{\dt}_{\ast}\dr\rightarrow B_{\ast}=A_{\ast}\dl \bfmath{x}_{\ast};\bfmath{\dt}_{\ast}\dr/\mathfrak{a}_{\ast}$
  is a stratified rational diagram (resp.\ rational diagram) modeled
  on a rooted tree $T$, and $S\subset T$ is a subtree consisting of
  descendants of a fixed vertex $v$, then the restriction
$$
\fR{t^{1/N_{\ast}}}|_S\rightarrow A_{\ast}|_S\rightarrow A_{\ast}\dl \bfmath{x}_{\ast};\bfmath{\dt}_{\ast}\dr|_S\rightarrow B_{\ast}|_S=A_{\ast}\dl \bfmath{x}_{\ast};\bfmath{\dt}_{\ast}\dr/\mathfrak{a}_{\ast}|_S
$$
on $S$ is a stratified rational diagram (resp.\ rational diagram)
modeled on $S$.
\end{rem}

\begin{prop}\label{prop-rationaldiagram3}
  Let $N$ be a multiple of all $N_v$'s.

  {\rm (1)} Consider the morphism
$$
\coprod_{\lambda}\Sp\mathcal{A}_{\lambda}[t^{1/N}]\longrightarrow\Sp\mathcal{A}_{v_0}[t^{1/N}],
$$
where $\lambda$ runs through all leaves of $T$.  Then any classical
point $\alpha\colon\Sp K\rightarrow\Sp\mathcal{A}_{v_0}[t^{1/N}]$
lifts to a classical point of the left-hand side.

{\rm (2)} Consider the morphism
$$
\coprod_{\lambda}\udp{\Sp(\mathcal{B}_{\lambda}[t^{1/N}])}\longrightarrow\udp{\Sp(\mathcal{B}_{v_0}[t^{1/N}])},
$$
where $\lambda$ runs through all leaves of $T$ and $\udp{(\cdot)}$
denotes the unit-polydisk part {\rm
  (\S\ref{subsub-unitpolydiskpart})}.  Then any classical point
$\alpha\colon\Sp K\rightarrow\udp{\Sp(\mathcal{B}_{v_0}[t^{1/N}])}$
valued in a subfield $F\subset K$ lifts to a classical point of the
left-hand side valued in $F$.
\end{prop}

\begin{proof}
  (1) It suffices to show that the morphism
$$
\coprod_{v\rightarrow w}\Sp\mathcal{A}_w[t^{1/N}]\longrightarrow\Sp\mathcal{A}_v[t^{1/N}],
$$
where $v$ is a vertex of $T$ and $w$ runs through all children of $v$,
has a similar property.  If the image of a classical point
$\alpha\colon\Sp K\rightarrow\Sp\mathcal{A}_v[t^{1/N}]$ lies outside
of the essential center, then the assertion follows from Lemma
\ref{lem-blow-up} (1).  If not, it lifts to the component
corresponding to a type 2 edge.

(2) The existence of the lifting can be shown similarly to (1). Note
that, if $w$ is a child of $v$, since the coordinate
$\bfmath{x}_w=(x_{w,1},\ldots,x_{w,n})$ is a $\ZZ$-rational coordinate
change of $\bfmath{x}_v=(x_{v,1},\ldots,x_{v,n})$, both
``unit-polydisk part'' and ``valued in $F$'' are preserved by the
morphism
$\Sp\mathcal{B}_w[t^{1/N}]\longrightarrow\Sp\mathcal{B}_v[t^{1/N}]$.
\end{proof}

\subsection{Noether normalization}
\begin{prop}\label{prop-normalization}
  Let
  $\bfmath{\dt}=(\dt_1,\ldots,\dt_n),\bfmath{\dt}'=(\dt'_1,\ldots,\dt'_n)\in\QQ^n$
  be vectors of rational numbers such that
  $0<\bfmath{\dt}'<\bfmath{\dt}$ $($i.e., $0<\dt'_i<\dt_i$ for
  $i=1,\ldots,n)$. Let $\mathcal{A}$ be an affinoid $\rR{t}$-integral
  domain, and
  $\mathcal{B}=\mathcal{A}\dl\bfmath{x};\bfmath{\dt}\dr/\mathfrak{a}$,
  where $\bfmath{x}=(x_1,\ldots,x_n)$, and $\mathfrak{a}$ is a prime
  ideal of $\mathcal{A}\dl\bfmath{x};\bfmath{\dt}\dr$.  Let $A$ be an
  admissible formal model of $\mathcal{A}$, and set
  $B=A\dl \bfmath{x};\bfmath{\dt}\dr/\udl{\mathfrak{a}}$, where
  $\udl{\mathfrak{a}}=\mathfrak{a}\cap A\dl\bfmath{x};\bfmath{\dt}\dr$,
  which is an admissible formal model of $\mathcal{B}$.  Then there
  exists a rational diagram
  $\fR{t^{1/N_{\ast}}}\longrightarrow A_{\ast}\rightarrow A_{\ast}\dl\bfmath{x}_{\ast};\bfmath{\dt}_{\ast}\dr\rightarrow B_{\ast}=A_{\ast}\dl\bfmath{x}_{\ast};\bfmath{\dt}_{\ast}\dr/\udl{\mathfrak{a}}_{\ast}$
  modeled on a rooted tree $T$ such that the following conditions are
  satisfied:
  \begin{itemize}
  \item[(a)] let $v_0\in\Ve_T$ be the root of $T$; then $A_{v_0}=A$,
    $\bfmath{x}_{v_0}=\bfmath{x}$, $\bfmath{\dt}_{v_0}=\bfmath{\dt}$,
    and $\udl{\mathfrak{a}}_{v_0}=\udl{\mathfrak{a}}$ $($hence
    $B_{v_0}=B)$;
  \item[(b)] for any leaf $\lambda$ of $T$, we have
    $\bfmath{\dt}'<\bfmath{\dt}_{\lambda}<\bfmath{\dt}$, and the map
    $A_{\lambda}\rightarrow B_{\lambda}$ factors as
$$
A_{\lambda}\longrightarrow A_{\lambda}\dl x_{\lambda,1},\ldots,x_{\lambda,d_{\lambda}};\dt_{\lambda,1},\ldots,\dt_{\lambda,d_{\lambda}}\dr\longhookrightarrow B_{\lambda},
$$
where $d_{\lambda}\geq 0$ is a non-negative integer, and the
right-hand map is injective and finite.
Moreover, one can take $\bfmath{\dt}_{\lambda}$ such that $\dt_{\lambda,1}=\cdots=\dt_{\lambda,d_{\lambda}}$.
\end{itemize}
\end{prop}

\begin{proof}
  The proposition follows by induction from the following:
  \begin{itemize}
  \item[$(\ast)$] If $\mathfrak{a}\neq(0)$, then for any
    $\bfmath{\dt}''\in(\QQ_{>0})^n$ such that
    $\bfmath{\dt}''<\bfmath{\dt}$, there exists a rational diagram as
    above satisfying the condition (a) and
    \begin{itemize}
    \item[(b$)'$] for any leaf $\lambda$ of $T$, we have
      $\bfmath{\dt}''<\bfmath{\dt}_{\lambda}<\bfmath{\dt}$ for all
      $i=1,\ldots,n$, and there exists a finite injection of the form
$$
A_{\lambda}\dl x_{\lambda,1},\ldots,x_{\lambda,n-1};\dt_{\lambda,1},\ldots,\dt_{\lambda,n-1}\dr/\mathfrak{a}'_{\lambda}\longhookrightarrow B_{\lambda},
$$
where $\dt_{\lambda,1}=\cdots=\dt_{\lambda,n-1}$; moreover, we have $\dim A_{\lambda}=A$ and $\dim B_{\lambda}=\dim B$.
\end{itemize}
\end{itemize}
Indeed, the proposition follows by repeated application of $(\ast)$,
each time replacing $A$ by $A_{\lambda}$ and $B$ by
$A_{\lambda}\dl y_1,\ldots,y_{n-1}\dr/\mathfrak{a}'_{\lambda}$, as
long as $\mathfrak{a}'_{\lambda}\neq (0)$.

We consider the {\em reversed lexicographical ordering} on $\NN^n$,
i.e., for $\nu=(\nu_1,\ldots,\nu_n),\mu=(\mu_1,\ldots,\mu_n)\in\NN^n$,
$\nu<\mu$ if and only if $\nu\neq\mu$ and $\nu_k<\mu_k$ where
$k=\max\{i\mid \nu_i\neq\mu_i\}$.

Suppose $\mathfrak{a}\neq(0)$.  Take a non-zero
$F\in\udl{\mathfrak{a}}=\mathfrak{a}\cap A\dl \bfmath{x};\bfmath{\dt}\dr$,
and write
$$
F(\bfmath{x})=\sum_{\nu\in\NN^n}c_{\nu}\bfmath{x}^{\nu},\quad c_{\nu}\in A.
$$
Consider the content ideal of $F$
$$
\Cont(F)=(c_{\nu}\mid\nu\in\NN^n).
$$
This is a non-zero ideal of a Noetherian ring
$A\dl\bfmath{x};\bfmath{\dt}\dr$, and is of the form
$(c_{\nu_0},\ldots,c_{\nu_s})$, where we suppose that
$\nu_0<\cdots<\nu_s$, and that the set of generators has been chosen
such that $\nu_s$ is the smallest possible.  Let us call such an
$\nu_s$ the {\em maximal content degree} of $F$, denoted by $\mcd(F)$.
The proof of $(\ast)$ will be done by induction on $\mcd(F)$, and
suppose that the claim is already proved for smaller $\mcd(F)$.

\medskip\noindent {\sc Step 1.}  We consider the formal blow-up
$\{A\rightarrow A_i\mid i=0,\ldots,s\}$ along the content ideal, i.e.,
$$
A_i=A\bigg\langle\!\!\!\bigg\langle\frac{c_{\nu_0}}{c_{\nu_i}},\ldots,\frac{c_{\nu_r}}{c_{\nu_i}}\bigg\rangle\!\!\!\bigg\rangle/\textrm{$c_{\nu_i}$-torsion}.
$$
For each $i=0,\ldots,s$, $F(\bfmath{x})$ in
$A_i\dl\bfmath{x};\bfmath{\dt}\dr$ is of the form $c_{\nu_i}F_i$ for
$F_i\in A_i\dl\bfmath{x};\bfmath{\dt}\dr$.  We have $c_{\nu_i}F_i=0$
in $B_i$, the strict transform of $B$ by $A\rightarrow A_i$, and since
$B_i$ is $c_{\nu_i}$-torsion free, we have $F_i=0$ in $B_i$, and so
$F_i\in\udl{\mathfrak{a}}A_i\dl\bfmath{x};\bfmath{\dt}\dr$.  For
$i<s$, $F_i(\bfmath{x})$ has a strictly smaller maximal content degree
since $c_{\nu_s}=(c_{\nu_s}/c_{\nu_i})c_{\nu_i}$, and hence the claim
$(\ast)$ for $A_i$ follows by induction.  The last patch $A_s$ will be
treated in the next step (with $A_s$ replaced by $A$).

\medskip\noindent {\sc Step 2.}  We discuss the case where
$\mcd(F)=\nu_0$ and $c_{\nu_0}=1$.  Take a natural number $M$ such
that
\begin{itemize}
\item $M$ is not divisible by the characteristic of $A$;
\item $M>|\nu_0|+1$ and $M^{M-1}>n|\nu_0|$;
\item if we set
$$
\bfmath{\e}=(1/M^{M^{n-1}},\ldots,1/M^{M^2},1/M^{M},1/M),
$$
then $\bfmath{\dt}''<\bfmath{\dt}-\bfmath{\e}$.
\end{itemize}
Set $N=M^{M^{n-1}}$ and replace $A$ by $A[t^{1/N}]$.  Note that
\begin{itemize}
\item[(i)] if $\nu_0<\mu$, then
  $\bfmath{\e}\cdot\nu_0<\bfmath{\e}\cdot\mu$.
\end{itemize}

Indeed, if $\mu=(\mu_1,\ldots,\mu_n)$ and
$\nu_0=(\nu_{0,1},\ldots,\nu_{0,n})$ with
$\mu_n=\nu_{0,n},\ldots,\mu_{k+1}=\nu_{0,k+1}$ and $\mu_k>\nu_{0,k}$
for $1\leq k\leq n$, then
\begin{equation*}
  \begin{split}
    \bfmath{\e}\cdot(\mu-\nu_0)&=M^{-M^{n-k}}(\mu_k-\nu_{0,k})+M^{-M^{n-k+1}}(\mu_{k+1}-\nu_{0,k+1})+\cdots+M^{-M^{n-1}}(\mu_1-\nu_{0,1})\\
    &\geqq M^{-M^{n-k}}-M^{-M^{n-k+1}}|\nu_0|-\cdots-M^{-M^{n-1}}|\nu_0|\\
    &\geqq M^{-M^{n-k}}\bigg\{1-\frac{(k-1)|\nu_0|}{M^{M-1}}\bigg\}\geqq M^{-M^{n-k}}\bigg\{1-\frac{n|\nu_0|}{M^{M-1}}\bigg\}> 0.
  \end{split}
\end{equation*}

We need to consider the scaling of the coordinate
$\til{\bfmath{x}}=t^{-\bfmath{\e}}\bfmath{x}$.  Note that
$A\dl\til{\bfmath{x}};\bfmath{\dt}\dr=A\dl\bfmath{x};\bfmath{\dt}-\bfmath{\e}\dr$
($\hookleftarrow A\dl\bfmath{x};\bfmath{\dt}\dr$).  If we set
$\til{F}(\til{\bfmath{x}})=F(\bfmath{x})$, we have
$$
\til{F}(\til{\bfmath{x}})=\sum_{\nu\in\NN^n}\til{c}_{\nu}\til{\bfmath{x}}^{\nu},\quad \til{c}_{\nu}=t^{\bfmath{\e}\cdot\nu}c_{\nu}.
$$

By (i) above, we have,
\begin{itemize}
\item[(ii)] if $\nu_0<\mu$, then
  $\til{c}_{\nu_0}=t^{\bfmath{\e}\cdot\nu_0}\ |\ t^{\bfmath{\e}\cdot\nu_0}c_{\mu}\ |\ t^{\bfmath{\e}\cdot\mu}c_{\mu}=\til{c}_{\mu}$;
  a fortiori, $\til{c}_{\mu}/\til{c}_{\nu_0}$ is divisible by
  $t^{1/N}$.
\end{itemize}

Indeed, $\til{c}_{\mu}/\til{c}_{\nu_0}$ is divisible by
$t^{\bfmath{\e}\cdot(\mu-\nu_0)}$, where
$\bfmath{\e}\cdot(\mu-\nu_0)>0$ by (i), and, in the notation as in the
proof of (i) above,
\begin{equation*}
  \begin{split}
    \bfmath{\e}\cdot(\mu-\nu_0)&=M^{-M^{n-1}}\{M^{M^{n-1}-M^{n-k}}(\mu_k-\nu_{0,k})+\cdots+(\mu_1-\nu_{0,1})\}\\
    &=\frac{1}{N}\cdot(\textrm{positive integer}).
  \end{split}
\end{equation*}

In particular, we have $\mcd(\til{F})\leqq\nu_0=\mcd(F)$.  Let
$$
\Cont(\til{F}(\til{\bfmath{x}}))=(\til{c}_{\mu_0},\ldots,\til{c}_{\mu_r})
$$
($\mu_0<\cdots<\mu_r$) be the content ideal of $\til{F}$ in $A'$,
where $\mu_r=\mcd(\til{F})$.  By (ii) we have $\mu_r\leqq\nu_0$.  If
$\mu_r<\nu_0$, then we can finish by induction.  So we assume
$\mu_r=\nu_0$.  We perform the formal blow-up along the content ideal
of $\til{F}(\til{\bfmath{x}})$, obtaining affine patches $\til{A}_j$
for $j=0,\ldots,r$ as before.  Similarly as before, we may finish by
induction for the patches $\til{A}_j$ for $j<r$.

On the patch $\til{A}_r$, the strict transform $\til{F}_r$ of
$\til{F}$ is of the form
$$
\til{F}_r(\til{\bfmath{x}})=f_r(\til{\bfmath{x}})+t^{1/N}G(\til{\bfmath{x}}),
$$
where $f_r(\til{\bfmath{x}})=\cdots+\til{\bfmath{x}}^{\nu_0}$ is a
monic polynomial and all terms in $G(\til{\bfmath{x}})$ are of degree
strictly larger than $\nu_0$.

Now we consider the famous coordinate change
$$
\begin{cases}
  x_n=y_n,\\
  x_k=y^{M^{n-k}}_n+y_k&(k=1,\ldots,n-1).
\end{cases}
$$
By our choice of $\bfmath{\e}$, the above coordinate change gives rise
to a coordinate change
$$
\begin{cases}
  \til{x}_n=\til{y}_n,\\
  \til{x}_k=\til{y}^{M^{n-k}}_n+\til{y}_k&(k=1,\ldots,n-1)
\end{cases}
$$
from $\til{\bfmath{x}}$ to
$\til{\bfmath{y}}=t^{-\bfmath{\e}}\bfmath{y}$, and we have
$$
\til{F}_r(\til{\bfmath{y}})=f_r(\til{\bfmath{y}})+t^{1/N}G(\til{\bfmath{y}})
$$
where
\begin{itemize}
\item[(iii)] $f_r(\til{\bfmath{y}})=u\til{y}^L_n+(\textrm{lower})$,
  where $u$ is a unit and
  $L=M^{n-1}\nu_{0,1}+M^{n-2}\nu_{0,2}+\cdots+\nu_{0,n}$;
\item[(iv)] all terms in $G(\til{\bfmath{y}})$ are of degree strictly
  larger than $\nu_0$.
\end{itemize}

Now let $B'_r$ be the strict transform of $B$ by
$A\rightarrow\til{A}_r$, and set
$\til{B}_r=B'_r\widehat{\otimes}_{\til{A}_r\dl\bfmath{y};\bfmath{\dt}\dr}\til{A}_r\dl\til{\bfmath{y}};\bfmath{\dt}\dr$.
Consider the surjection
$$
\til{A}_r\dl\til{\bfmath{y}};\bfmath{\dt}\dr\longrightarrow\til{B}_r
$$
whose kernel is
$\udl{\mathfrak{a}}_r:=\udl{\mathfrak{a}}\til{A}_r\dl\til{\bfmath{y}};\bfmath{\dt}\dr$.
Let $\til{B}'_r$ be the image of
$$
\til{A}_r\dl\til{\bfmath{y}}';\til{\bfmath{\dt}}'\dr:=\til{A}_r\dl \til{y}_1,\ldots,\til{y}_{n-1};\dt_1,\ldots,\dt_{n-1}\dr\longrightarrow\til{B}_r.
$$
We have $\til{B}_r=\til{B}'_r\dl\til{y}_n;\dt_n\dr$.  Note that
$\til{A}_r\dl\til{\bfmath{y}};\bfmath{\dt}\dr=\til{A}_r\dl\bfmath{y};\bfmath{\dt}-\bfmath{\e}\dr$,
$\til{A}_r\dl\til{\bfmath{y}}';\til{\bfmath{\dt}}'\dr=\til{A}_r\dl\bfmath{y}';\bfmath{\dt}'-\bfmath{\e}'\dr$,
where $\bfmath{\e}'=(\e_1,\ldots,\e_{n-1})$, and hence
$\til{B}_r=\til{B}'_r\dl y_n;\dt_n-\e_n\dr$.

We need to show that $\til{B}_r$ is finite over $\til{B}'_r$.  It
suffices to show that $(\til{B}_r)_0=\til{B}_r/t^{1/N}\til{B}_r$ is
finite over $(\til{B}'_r)_0=\til{B}'_r/t^{1/N}\til{B}'_r$
(\cite[8.4]{Matsu}).  But this follows from (iii) and (iv).

To conclude the proof, let $\lambda$ be a leaf, and
$A_{\lambda}=A_r$, $\bfmath{x}_{\lambda}=\bfmath{y}$,
$\bfmath{\dt}_{\lambda}=\bfmath{\dt}-\bfmath{\e}$, and
$B_{\lambda}=\til{B}_r$.
Moreover, we claim that we can put $\dt_{\lambda,1}=\cdots=\dt_{\lambda,n-1}$.
To this end, define $\dt=\min\{\dt_{\lambda,i}\mid i=1,\ldots,n-1\}$, and let $\bfmath{x}'_{\lambda}=(x_{\lambda,1},\ldots,x_{\lambda,n-1})$ and $\bfmath{\dt}'_{\lambda}=(\dt_{\lambda,1},\ldots,\dt_{\lambda,n-1})$.
We take the base change of the finite injective map
$$
A_{\lambda}\dl \bfmath{x}'_{\lambda};\bfmath{\dt}'_{\lambda}\dr/\mathfrak{a}'_{\lambda}\longhookrightarrow B_{\lambda}=A_{\lambda}\dl \bfmath{x}_{\lambda};\bfmath{\dt}_{\lambda}\dr/\mathfrak{a}_{\lambda}\cong(A_{\lambda}\dl\bfmath{x}'_{\lambda};\bfmath{\dt}'_{\lambda}\dr\widehat{\otimes}_{A_{\lambda}}A_{\lambda}\dl x_n;\dt_{\lambda,n}\dr)/\mathfrak{a}_{\lambda}\eqno{(\ast)}
$$
by $A_{\lambda}\dl\bfmath{x}'_{\lambda};\bfmath{\dt}'_{\lambda}\dr\rightarrow A_{\lambda}\dl\bfmath{x}'_{\lambda};\dt\dr$, followed by killing $t$-torsions.
Here, by an abuse of notation, $\dt$ in the right-hand side stands for the vector $(\delta,\ldots,\delta)$ of size $n-1$.
We claim that the resulting map
$$
A_{\lambda}\dl\bfmath{x}'_{\lambda};\dt\dr/\mathfrak{a}''_{\lambda}\longrightarrow B'_{\lambda}\cong A_{\lambda}\dl\bfmath{x}'_{\lambda},x_n;\dt,\dt_{\lambda,n}\dr/\mathfrak{a}'''_{\lambda}
$$
is again finite and injective.
Finiteness is obvious.
To show it is injective, it suffices to show that its localization by $1/t$
$$
\cA_{\lambda}\dl\bfmath{x}'_{\lambda};\dt\dr/\mathfrak{a}''_{\lambda}\cA_{\lambda}\dl\bfmath{x}'_{\lambda};\dt\dr\longrightarrow \cB'_{\lambda}\eqno{(\dast)}
$$
is injective.
Consider the base change diagram 
$$
\xymatrix{\cA_{\lambda}\dl\bfmath{x}'_{\lambda};\dt\dr/\mathfrak{a}''_{\lambda}\cA_{\lambda}\dl\bfmath{x}'_{\lambda};\dt\dr\ar[r]&\cB_{\lambda}\otimes_{\cA_{\lambda}\dl\bfmath{x}'_{\lambda};\bfmath{\dt}'_{\lambda}\dr}\cA_{\lambda}\dl\bfmath{x}'_{\lambda};\dt\dr\\ \cA_{\lambda}\dl \bfmath{x}'_{\lambda};\bfmath{\dt}'_{\lambda}\dr/\mathfrak{a}'_{\lambda}\cA_{\lambda}\dl \bfmath{x}'_{\lambda};\bfmath{\dt}'_{\lambda}\dr\ar[u]\ar@{^{(}->}[r]&\cB_{\lambda}\rlap{,}\ar[u]}
$$
by $\cA_{\lambda}\dl\bfmath{x}'_{\lambda};\bfmath{\dt}'_{\lambda}\dr\rightarrow\cA_{\lambda}\dl\bfmath{x}'_{\lambda};\dt\dr$, where the lower horizontal arrow is the localization by $1/t$ of $(\ast)$, which is finite and injective.
The upper horizontal arrow is injective, since $\cA_{\lambda}\dl\bfmath{x}'_{\lambda};\bfmath{\dt}'_{\lambda}\dr\rightarrow\cA_{\lambda}\dl\bfmath{x}'_{\lambda};\dt\dr$ is flat due to Lemma \ref{lem-openimmersion}.
Moreover, by finiteness of the horizontal arrows, $\cB_{\lambda}\otimes_{\cA_{\lambda}\dl\bfmath{x}'_{\lambda};\bfmath{\dt}'_{\lambda}\dr}\cA_{\lambda}\dl\bfmath{x}'_{\lambda};\dt\dr$ is already complete, and hence coincides with $\cB'_{\lambda}$, i.e., the upper horizontal arrow coincides with $(\dast)$, whence the desired injectivity.
Now, we finish the proof by replacing $A_{\lambda}\dl \bfmath{x}_{\lambda};\bfmath{\dt}_{\lambda}\dr$ by $A_{\lambda}\dl \bfmath{x}'_{\lambda},x_{\lambda,n};\dt,\dt_{\lambda,n}\dr$ and $B_{\lambda}$ by $B'_{\lambda}$.
\end{proof}

Let us call the rational diagram
$\fR{t^{1/N_{\ast}}}\rightarrow A_{\ast}\rightarrow A_{\ast}\dl\bfmath{x}_{\ast};\bfmath{\dt}_{\ast}\dr\rightarrow B_{\ast}=A_{\ast}\dl\bfmath{x}_{\ast};\bfmath{\dt}_{\ast}\dr/\mathfrak{a}_{\ast}$
modeled on $T$ as in Proposition \ref{prop-normalization} a {\em
  normalization diagram}.  Note that, if $S\subset T$ is a subtree of
descendants of a fixed vertex, then
$\fR{t^{1/N_{\ast}}}|_S\rightarrow A_{\ast}\dl\bfmath{x}_{\ast};\bfmath{\dt}_{\ast}\dr|_S\rightarrow B_{\ast}|_S=A_{\ast}\dl\bfmath{x}_{\ast};\bfmath{\dt}_{\ast}\dr/\mathfrak{a}_{\ast}|_S$
is again a normalization diagram.

\begin{thm}[Stratified Noether normalization]\label{thm-stratifiednormalization}
  Let
  $\bfmath{\dt}=(\dt_1,\ldots,\dt_n),\bfmath{\dt}'=(\dt'_1,\ldots,\dt'_n)\in\QQ^n$
  be vectors of rational numbers such that
  $0<\bfmath{\dt}'<\bfmath{\dt}$ $($i.e., $0<\dt'_i<\dt_i$ for
  $i=1,\ldots,n)$. Let $\mathcal{A}$ be an affinoid $\rR{t}$-algebra,
  and
  $\mathcal{B}=\mathcal{A}\dl\bfmath{x};\bfmath{\dt}\dr/\mathfrak{a}$,
  where $\bfmath{x}=(x_1,\ldots,x_n)$ and $\mathfrak{a}$ is an ideal
  of $\mathcal{A}\dl\bfmath{x};\bfmath{\dt}\dr$.  Let $A$ be an
  admissible formal model of $\mathcal{A}$, and set
  $B=A\dl \bfmath{x};\bfmath{\dt}\dr/\udl{\mathfrak{a}}$, where
  $\udl{\mathfrak{a}}=\mathfrak{a}\cap A\dl\bfmath{x};\bfmath{\dt}\dr$,
  which is an admissible formal model of $\mathcal{B}$.  Then there
  exists a stratified rational diagram
  $\fR{t^{1/N_{\ast}}}\rightarrow A_{\ast}\rightarrow A_{\ast}\dl\bfmath{x}_{\ast};\bfmath{\dt}_{\ast}\dr\rightarrow B_{\ast}=A_{\ast}\dl\bfmath{x}_{\ast};\bfmath{\dt}_{\ast}\dr\mathfrak{a}_{\ast}$
  modeled on a rooted tree $T$ such that the conditions {\rm (a)} and
  {\rm (b)} in Proposition {\rm \ref{prop-normalization}} are
  satisfied.
\end{thm}

\begin{proof} First perform the formal blow-up of $A$ along $(0)$,
  which gives the empty set of type $1$ edges, and the type $2$ edges
  corresponds to the irreducible decomposition of $\Sp\mathcal{B}$.
  Then one can construct the desired stratified normalization diagram
  by the normalization diagrams as in Proposition
  \ref{prop-normalization} starting from the irreducible components,
  joining at each vertex the normalization diagrams from irreducible
  components of the induced families over the essential center of the
  formal blow-ups.
\end{proof}

The stratified rational diagram as above will be called the {\em
  stratified normalization diagram}.  Note that, if $S\subset T$ is a
subtree of descendants of a fixed vertex, then
$A_{\ast}|_S\rightarrow A_{\ast}\dl\bfmath{x}_{\ast};\bfmath{\dt}_{\ast}\dr|_S\rightarrow B_{\ast}|_S=A_{\ast}\dl\bfmath{x}_{\ast};\bfmath{\dt}_{\ast}\dr/\mathfrak{a}_{\ast}|_S$
is again a stratified normalization diagram.

\section{Interpolation of algebraic points}\label{sec-interpolation}
\subsection{Interpolation determinants}\label{sub-interpolationdet}
\subsubsection{Situation}\label{subsub-sit-interpolation}

Throughout this section, we fix $\delta>0$, an admissible
$\fR{t}$-algebra $A$, and $B=A\dl\bfmath{x};\delta\dr/\mathfrak{a}$,
where $\bfmath{x}=(x_1,\ldots,x_n)$ and $\delta$ stands for the row vector
$(\delta,\ldots,\delta)$.  Moreover, we assume that there
exists a factoring map
$$
A\longrightarrow A\dl x_1,\ldots,x_d;\delta\dr\stackrel{\phi}{\longrightarrow}B\eqno{(\star)}
$$
such that
\begin{itemize}
\item $d\leq n$; 
\item $\phi$ is finite with nilpotent kernel. 
\end{itemize}

We set $\mathcal{A}=A[1/t]$ and $\mathcal{B}=B[1/t]$, and fix a
minimal generators $v_1,\ldots,v_E$ of $B$ as an
$A\dl x_1,\ldots,x_d;\delta\dr$-module.

We moreover fix a $K$-valued point
$\alpha\colon\Sp K\rightarrow\Sp\cA$, where $K$ is a non-archimedean local field finite over $\rR{t}$, and let
$\Sp\cB_{\alpha}\rightarrow\Sp K$ be the fiber of
$\Sp\cB\rightarrow\Sp\cA$ over $\alpha$, i.e.,
$\cB_{\alpha}=\cB\widehat{\otimes}_{\cA}K\ (=\cB\otimes_{\cA}K)$, which is an affinoid
algebra over $K$.  Note that $\mathcal{B}_{\alpha}$ has
$B_{\alpha}=B\widehat{\otimes}_AV_K/(\textrm{$t$-torsion})\ (=B\otimes_AV_K/(\textrm{$t$-torsion}))$ as an
admissible formal model, where $V_K$ is the valuation ring of $K$.  We
have the base change of $(\star)$
$$
V_K\longrightarrow V_K\dl x_1,\ldots,x_d;\delta\dr\stackrel{\phi_{\alpha}}{\longrightarrow} B_{\alpha},\eqno{(\star)_{\alpha}}
$$
where $\phi_{\alpha}$ is finite with nilpotent kernel due to Lemma
\ref{lem-finiteinjection}.

Note that $K$ has the unique norm $|\cdot|$ such that $|t|=e^{-1}$,
where $e>1$ is the real number fixed in \S\ref{sub-norms}.

\subsubsection{Interpolation determinants}\label{subsub-interpolationdet}
Let $\alpha_i\colon\Sp K\rightarrow\Sp\cB_{\alpha}$ for
$i=1,\ldots,\mu$ be $K$-valued points over $\Sp K$, i.e., sections of $\Sp\cB_{\alpha}\rightarrow\Sp K$, and
$p_i\colon B_{\alpha}\rightarrow V_K$ the $V_K$-algebra homomorphisms corresponding to
$\alpha_i$.

Set $\bfmath{p}=(p_1,\ldots,p_{\mu})$.  For
$\bfmath{f}=(f_1,\ldots,f_{\mu})\in B^{\mu}_{\alpha}$,
we consider the {\em interpolation determinant}
$$
\Delta(\bfmath{f},\bfmath{p})=\det\begin{bmatrix}p_j(f_i))\mid i,j=1,\ldots,\mu\end{bmatrix}.
$$
We want to estimate the norm of $\Delta(\bfmath{f},\bfmath{p})$ in
terms of
$$
\rho=\rho(\bfmath{p})=|t^\delta|\cdot\max_{{j,j'=1,\ldots,\mu}\atop {i=1,\ldots,d}}\begin{vmatrix}p_j(x_i)-p_{j'}(x_i)\end{vmatrix}
$$
i.e., the radius of the smallest ball containing the projections of
the points $\alpha_1,\ldots,\alpha_{\mu}$ to the $y$-coordinates,
normalized by a factor of $|t^\delta|$. The following estimate of
$|\Delta(\bfmath{f},\bfmath{p})|$ is the key to our approach, which
goes back to an idea by Bombieri-Pila \cite{bp}.
Cluckers-Comte-Loeser \cite{ccl:pw} have used a similar
non-archimedean analogue in their work using smooth parametrizations
\cite{ccl:pw}, but we require a somewhat different statement as we
replace parametrizing maps by finite modules.  In the complex analytic
context, a similar argument using finite modules was used in
Binyamini-Novikov \cite{bn:analytic-pw}.

\begin{prop}\label{prop-interpdet}
  We have
$$
\big|\Delta(\bfmath{f},\bfmath{p})\big|\leq \rho^{C_dE^{-1/d}\mu^{1+1/d}},
$$
where $C_d$ is a positive constant depending only on $d$.  $($Recall
that $E$ is the number of minimal generators of $B$ as an
$A\dl x_1,\ldots,x_d;\delta\dr$-module; cf.\ {\rm
  \S\ref{subsub-sit-interpolation}}$).$
\end{prop}

\begin{proof}
  We first claim that we may assume, without loss of generality, that
  $$
  \rho=|t^\delta|\cdot\max_{{j=1,\ldots,\mu}\atop {i=1,\ldots,d}}\begin{vmatrix}p_j(x_i)\end{vmatrix}.\eqno{(\ast)}
  $$
  To do this, we consider an automorphism (parallel translation)
  $\Phi$ of $V_K\dl x_1,\ldots,x_d;\delta\dr$ of the form
  $\Phi(x_i)=x_i-s_i$ for $i=1,\ldots,d$, where $s_i\in V_K$ are
  chosen such that $(s_i=)$ $p_1(s_i)=p_1(x_i)$.  Pulling back by
  $\Phi$, we have similar situation where $B_{\alpha}$ is replaced by
  an isomorphic copy
  $\Phi^{\ast}B_{\alpha}=B_{\alpha}\otimes_{V_K\dl x_1,\ldots,x_d;\delta\dr,\Phi}V_K\dl x_1,\ldots,x_d;\delta\dr$.
  Note that, on $\Phi^{\ast}B_{\alpha}$, we have
  $p_1(x_i)=0$ for $i=1,\ldots,d$.  Then the desired inequality is
  equivalent to the corresponding one on
  $\Phi^{\ast}B_{\alpha}$, in which the $\rho$ is
  given as in $(\ast)$.

  For $i=1,\ldots,\mu$, write
  $$
  f_i=\sum^E_{k=1}f_{ik}v_k,\quad f_{ik}\in V_K\dl x_1,\ldots,x_d;\delta\dr,
  $$
  and further expand
  $$
  f_{ik}=\sum_{\nu\in\NN^d}c_{i,k,\nu}\bfmath{y}^{\nu},\quad c_{i,k,\nu}\in t^{\nu\cdot\delta}V_K,
  $$
where $\bfmath{y}=(x_1,\ldots,x_d)$.
  We now expand $\Delta(\bfmath{f},\bfmath{p})$, by multilinearity
  with respect to each column, as a sum of interpolation determinants
  of the form $\Delta(\bfmath{f}',\bfmath{p})$ where
  $f'_i=c_{i,k_i,\nu_i}\bfmath{y}^{\nu_i}v_{k_i}$, for some choice of
  $k_i,\nu_i$ for each $i=1,\ldots,\mu$.  It will be enough to
  estimate each of such $\Delta(\bfmath{f}',\bfmath{p})$ separately.
  
  Now comes the main idea of Bombieri-Pila: if for two different
  columns corresponding to $f'_q$ and $f'_r$ one makes the same choice
  $k=k_q=k_r$ and $\nu=\nu_q=\nu_r$, then
  $$
  \bfmath{p}(f'_q)=c_{q,k,\nu}\cdot\bfmath{p}(\bfmath{y}^{\nu}v_k),\quad \bfmath{p}(f'_r)=c_{r,k,\nu}\cdot\bfmath{p}(\bfmath{y}^{\nu}v_k),
  $$
  and the corresponding columns thus agree up to a constant from $K$,
  hence the determinant vanishes.
  
  So, in order for the determinant $\Delta(\bfmath{f}',\bfmath{p})$ to
  be non-zero, one has to choose each pair $(k_i,\nu_i)$ at most one
  $c_{i,k_i,\nu_i}\bfmath{y}^{\nu_i}v_{k_i}$.  For any such choice,
  and any $p_j$, we have
  $$
  |p_j(c_{i,k_i,\nu_i}\bfmath{y}^{\nu_i}v_{k_i})|\le |t|^{\delta\nu_i} \cdot|p_j(\bfmath{y}^{\nu_i})| \le \rho^{|\nu_i|},
  $$
  and therefore
  $$
  \big|\Delta(\bfmath{f}',\bfmath{p})\big|\leq\rho^{\sum_{i=1,\ldots,\mu}|\nu_i|}.\eqno{(\dast)}
  $$
  It is clear that the largest value for the right-hand side will be
  obtained if we choose as many monomials $\bfmath{y}^{\nu_i}$ as possible with
  $|\nu_i|=0$, then as many as possible with $|\nu_i|=1$, etc.  There
  are at most $\Theta_d(k^{d-1})$ monomials $\bfmath{y}^{\nu_i}$ with
  $|\nu_i|=k$, and each can appear at most $E$ times for different
  choices of $k_i=1,\ldots,E$.  Since the number of summands is $\mu$,
  we solve
  $$
  \sum^N_{k=1}k^{d-1}E=\mu\quad\Longrightarrow\quad N=\Omega_d\bigg(\frac{\mu}{E}\bigg)^{1/d},
  $$
  and deduce that the largest possible
  $\big|\Delta(\bfmath{f}',\bfmath{p})\big|$ will include all possible
  summands up to order $|\nu|=\lfloor N\rfloor$.  In this case, we
  have
  $$
  \sum_{i=1}^{\mu}|\nu_i|\geq\sum_{k=1}^Nk^dE=\Omega_d(E\mu^{d+1})=\Omega_d(E^{-1/d}\mu^{1+1/d}),
  $$
  which, combined with $(\dast)$, yields our assertion.
\end{proof}

\subsection{Polynomial interpolation
  determinants}\label{sub-polyinterpdet}

\subsubsection{Situation}\label{subsub-sit-polyinterpolation}
We continue with working in the situation as in
\S\ref{subsub-sit-interpolation}, and we further make the following
setup.

We assume that the $K$-valued point
$\alpha\colon\Sp K\rightarrow\Sp\cA$ lies over a fixed $F$-valued point
$\beta\colon\Sp F\rightarrow\Sp\rR{t}$, i.e., $F$ is a non-archimedean
local subfield of $K$ such that $K/F$ is a finite extension, and
$$
\xymatrix{\Sp K\ar[r]^-{\alpha}\ar[d]&\Sp\cA\ar[d]\\ \Sp F\ar[r]_-{\beta}&\Sp\rR{t}}
$$
is commutative. We denote by $\pi\in V_F$ the
uniformizer of $F$, and by $q=p^f$ the number of elements of the residue
field $V_F/(\pi)$.  
Note that $F$ has the unique norm $|\cdot|$ such that $|t|=e^{-1}$,
where $e>1$ the real number fixed in \S\ref{sub-norms}, and this norm
coincides with the restriction of the norm on $K$ as in
\S\ref{subsub-sit-interpolation}.  
We have $t=u\pi^r$ by $u\in V^{\times}_F$ and $r\geq 1$.
In particular, we have $|t|\leq|\pi|$.

Note that, even though we are interested in counting $F$-valued points of the affinoid spaces, we need to consider a larger field $K$ to capture them, since we occasionally take base change by $\rR{t}\rightarrow\rR{t^{1/N}}$.

Let $\mu=\mu(D)$ be the dimension of the space of polynomials of
degree at most $D\in\NN$ in $d+1$ variables.  We fix $d+1$ functions
$\bfmath{f}=(f_1,\ldots,f_{d+1})\in B^{d+1}_{\alpha}$
and $\mu$-tuple of $K$-valued points $\bfmath{p}=(p_1,\ldots,p_{\mu})$
(with $p_i\colon B_{\alpha}\rightarrow V_K$), as before.  We assume
\begin{itemize}
\item $p_j(f_i)\in F$ for $j=1,\ldots,\mu$ and $i=1,\ldots,d+1$;
\item $p_j(x_i)\in F$ for $j=1,\ldots,\mu$ and $i=1,\ldots,d$.
\end{itemize}

We define the {\em polynomial interpolation determinant}
$\Delta^D(\bfmath{f},\bfmath{p})$ to be
$$
\Delta^D(\bfmath{f},\bfmath{p})=\Delta(\bfmath{g},\bfmath{p}),\quad \bfmath{g}=\big(\bfmath{f}^{\nu}\mid\nu\in\NN^{d+1},|\nu|\leq D\big).
$$

In the following, we will denote by $|\cdot|_{\infty}$ the archimedean norm in order to distinguish from the norm $|\cdot|$.
The following lemma follows by elementary linear algebra.
\begin{lem}\label{lem-polyinterpdet}
  Let $P$ be a collection of adic morphisms
  $p\colon B_{\alpha}\rightarrow V_K$ over $V_K$ such that $p(f_i)\in F$
  for $i=1,\ldots,d+1$, and suppose that
  $\Delta^D(\bfmath{f},\bfmath{p})$ vanishes for every $\mu$-tuple
  $\bfmath{p}$ of points from $P$.  Then the set of points
$$
\{p(\bfmath{f})\mid p\in P\}\subset F^{d+1}
$$
is contained in a hypersurface defined by a polynomial
$Q\in F[x_1,\ldots,x_{d+1}]$ of degree at most $D$.
\end{lem}

Now, let $H\in\NN$ be an arbitrary natural number and set $h=\log_q H$.

\begin{prop}\label{prop-polyinterpdet}
  Suppose $H(p_i(f_j))\leq H$ for $i=1,\ldots,\mu$ and
  $j=1,\ldots,d+1$.  Then either $\Delta^D(\bfmath{f},\bfmath{p})=0$
  or
$$
\big|\Delta^D(\bfmath{f},\bfmath{p})\big|\geq\begin{cases}|t|^{([F:\QQ_p]/r)\{\log_q(\mu!)+(d+2)D\mu h\}}&(F\supset\QQ_p),\\ |t|^{(d+2)D\mu h}&(F\supset\FF_q\dpl t\dpr).\end{cases}
$$
\end{prop}

\begin{proof}
Suppose that $\Delta^D(\bfmath{f},\bfmath{p})\neq 0$, and let us first treat the case where $F$ is a finite extension of $\QQ_p$, so that
  $p_j(f_i)=a_{ij}/b_{ij}$ are rational numbers and $|a_{ij}|_{\infty},|b_{ij}|_{\infty}\leq H$.
  Set $p=v\pi^s$ for $v\in V^{\times}_F$ and $s\geq 1$, so that $|q|^r=|t|^{fs}=|t|^{[F:\QQ_p]}$.
All numbers in the column of $\Delta^D(\bfmath{f},\bfmath{p})$
  corresponding to $p_j$ have denominators dividing
  $b_j=\prod_i b_{ij}^D$, and setting $b=\prod_j b_j$ we have
  $|b|_{\infty}\le H^{(d+1)D\mu}$.  Since every entry in the matrix defining
  $\Delta^D(\bfmath{f},\bfmath{p})$ is bounded in the classical
  absolute value by $H^D$, we also have
  $$
  |\Delta^D(\bfmath{f},\bfmath{p})|_\infty \le \mu! H^{D\mu}.
  $$
  Then $b\Delta^D(\bfmath{f},\bfmath{p})\in\ZZ$ and
  $$
  |b\Delta^D(\bfmath{f},\bfmath{p})|_\infty \le \mu! H^{(d+2)D\mu},
  $$
  and therefore
  $$
  |\Delta^D(\bfmath{f},\bfmath{p})| \ge |b\Delta^D(\bfmath{f},\bfmath{p})|\ge |q|^{\log_q |b\Delta^D(\bfmath{f},\bfmath{p})|_\infty}=|t|^{(fs/r)\{\log_q(\mu!)+(d+2)D\mu h\}}.
  $$
  
  Simiarly, in the $F=\FF_q\dpl\pi\dpr$ case, all rational functions $p_j(f_i)=a_{ij}/b_{ij}\in\FF_q(t)$ in
  the column of $\Delta^D(\bfmath{f},\bfmath{p})$ corresponding to
  $p_j$ have denominators dividing $b_j$ as above, and $b$ has degree
  bounded by $(d+1)Dh$.  Since every entry in the matrix defining
  $\Delta^D(\bfmath{f},\bfmath{p})$ is a rational function of degree
  at most $Dh$, we also have
  $$
  \deg_t \Delta^D(\bfmath{f},\bfmath{p}) \le D\mu h.
  $$
  Then $b\Delta^D(\bfmath{f},\bfmath{p})\in\FF_q[t]$ and
  $$
  \deg_t b\Delta^D(\bfmath{f},\bfmath{p}) \le (d+2)D\mu h,
  $$
  and therefore
  $$
  |\Delta^D(\bfmath{f},\bfmath{p})| \ge |b\Delta^D(\bfmath{f},\bfmath{p})| \ge \deg b\Delta^D(\bfmath{f},\bfmath{p}) \ge |t|^{(d+2)D\mu h},
  $$
  as claimed.
\end{proof}

\subsection{Interpolation by a hypersurface}\label{sub-interpolationhyp}
Let $P$ be as in Lemma \ref{lem-polyinterpdet}, 
and let
$$
\rho=\rho(P)=|t|^\delta\cdot \max_{{p\neq p'\in P}\atop {i=1,\ldots,d}}\begin{vmatrix}p(x_i)-p'(x_i)\end{vmatrix}.
$$

Comparing Proposition~\ref{prop-interpdet} and
Proposition~\ref{prop-polyinterpdet} we obtain the following
corollary.

\begin{cor}\label{cor-hypersurface-select}
  Suppose $H(p(f_i))\leq H$ for $i=1,\ldots,d+1$ and $p\in P$.  Let
  $\varepsilon>0$.  There exist a constant 
  $$
  D=\begin{cases}D(d,\varepsilon,E,[F:\QQ_p])&(F\supset\QQ_p),\\ D(d,\varepsilon,E,r)&(F\supset\FF_q\dpl t\dpr),\end{cases}
  $$
  which does not depend on $\alpha\colon\Sp K\rightarrow\Sp\cA$, such that
  if $\rho<|\pi|^{h\varepsilon}$ then the set of points
$$
\{p(\bfmath{f})\mid p\in P\}\subset F^{d+1}
$$
is contained in a hypersurface defined by a polynomial
$Q\in F[X_1,\ldots,X_{d+1}]$ of degree at most $D$.
\end{cor}

\begin{proof}
  By Lemma~\ref{lem-polyinterpdet} it is enough to prove that
  $\Delta^D(\bfmath{f},\bfmath{p})=0$ for every $\mu$-tuple
  $\bfmath{p}\subset P$. 
  Assume the contrary. 
  In the $F\supset\QQ_p$ case, we have by Proposition~\ref{prop-polyinterpdet} a lower bound
  $$
  \big|\Delta^D(\bfmath{f},\bfmath{p})\big|\geq|t|^{([F:\QQ_p]/r)\{\log_q(\mu!)+(d+2)D\mu h\}}.
  $$
  On the other hand, by Proposition~\ref{prop-interpdet} we have an
  upper bound
  $$
  \big|\Delta(\bfmath{f},\bfmath{p})\big|\leq \rho^{C_dE^{-1/d}\mu^{1+1/d}}.
  $$
  Thus
  $$ 
  \rho^{C_dE^{-1/d}\mu^{1+1/d}} \ge |t|^{([F:\QQ_p]/r)\{\log_q(\mu!)+(d+2)D\mu h\}}. 
  $$
  Taking $\mu\sim_d D^{d+1}$ and $\rho<|\pi|^{\varepsilon h}=|t|^{\e h/r}$ into account, 
  we have a contradiction as soon as
  $$
  C'_d h\e E^{-1/d} D^{d+2+\frac{1}{d}} \ge [F:\QQ_p]\{(d+1)\log_q D+(d+2)C''_d D^{d+2} h\}
  $$
  with suitable constants $C'_d,C''_d$. Thus we obtain a
  contradiction, for instance with a suitable
  $D\sim_d E\cdot ([F:\QQ_p]/\varepsilon)^d$.
  
  In the $F\supset\FF_q\dpl t\dpr$ case, we have
  $$
  \rho^{C_dE^{-1/d}\mu^{1+1/d}} \ge |t|^{(d+2)D\mu h}.
  $$
  Similarly to the previous case, 
  we have a contradiction as soon as
  $$
  C'_d \e E^{-1/d} D^{d+2+\frac{1}{d}} \ge (d+2)C''_d D^{d+2} r
  $$
  with suitable constants $C'_d,C''_d$. Thus we obtain a
  contradiction, for instance with a suitable
  $D\sim_d E\cdot (r/\varepsilon)^d$.
\end{proof}

\begin{prop}\label{prop-hypersurfaces}
  Let $H\in\NN$ and $\varepsilon>0$.  There exist a constant
  $$
  D'=\begin{cases}D'(d,\varepsilon,E,\delta,[F:\QQ_p])&(F\supset\QQ_p),\\ D'(d,\varepsilon,E,\delta,r)&(F\supset\FF_q\dpl t\dpr),\end{cases}
  $$
  not depending on $\alpha\colon\Sp K\rightarrow\Sp\cA$, and a collection of
  $H^\varepsilon$ hypersurfaces in $F^{d+1}$, each of
  degree at most $D'$, such that the following holds: for every
  $K$-valued point $p:\cB_{\alpha}\to K$ over $\alpha$ such
    that $p(f_i)\in F$ and $H(p(f_i))\le H$ for $i=1,\dots,d+1$, the
  point $p(\bfmath{f})$ belongs to one of the hypersurfaces.
\end{prop}

\begin{proof}
  Suppose first that $h\e/(2d)\ge1$. Apply
  Corollary~\ref{cor-hypersurface-select} with $\e/(2d)$ in place of
  $\e$, and denote by $D'$ the resulting degree. It will suffice to
  subdivide the collection of points $p:B_{\alpha}\to V_K$ as above into
  $H^{\varepsilon}=q^{h\varepsilon}$ subcollections $P_q$ satisfying
  $\rho(P_q)\le|\pi|^{h\varepsilon/(2d)}$.  For each $j$ the residues
  $p(x_j)\in V_F/(\pi)^{\lceil h\varepsilon/(2d)\rceil}$ take at most
  $q^{\lceil h\varepsilon/(2d)\rceil}$ different values.  Defining
  each collection to consist of points that give the same residue for
  each $j$ we obtain
  $$
  q^{d\lceil h\e/(2d)\rceil} < q^{d \cdot h\e/d} \le H^{\e}
  $$
  such collections, as required.

  Now suppose $h\e/(2d)<1$. Apply
  Corollary~\ref{cor-hypersurface-select} with $\delta\e/(2d)$ in
  place of $\e$ and denote by $D'$ the resulting degree. Then the
  collection of all points $p:B_{\alpha}\to V_K$ as above satisfies
  $$
  \rho(P) \le |t|^\delta \le |\pi|^{h\delta\e/(2d)}
  $$
  so a single hypersurface of degree $D'$ suffices to interpolate al
  points in $P$.
\end{proof}

We also prove a polylogarithmic interpolation result, as follows.

\begin{prop}\label{prop-hypersurfaces-pl}
  Let $H\in\NN$ and $\varepsilon>0$. There exist a constant
  $$
  C=\begin{cases}C(d,E,\delta,[F:\QQ_p]/r)&(F\supset\QQ_p),\\ C(d,E,\delta)&(F\supset\FF_q\dpl t\dpr),\end{cases}
  $$
  not
  depending on $\alpha\colon\Sp K\rightarrow\Sp\cA$, and a
  hypersurface $V\subset K^{d+1}$ of degree at most
  $C\cdot h^{d}$, such that the following holds: for every
  $K$-valued point $p:\cB_{\alpha}\to K$ over $\alpha$ such that
  $p(f_j)\in F$ and $H(p(f_j))\le H$ for $j=1,\dots,d+1$ we have
  $p(\bfmath{f})\in V$.
\end{prop}

\begin{proof}
  By Lemma~\ref{lem-polyinterpdet} it is enough to prove that
  $\Delta^D(\bfmath{f},\bfmath{p})=0$ for every $\mu$-tuple
  $\bfmath{p}$ of points as above. 
  Assume the contrary. 
  In the $F\supset\QQ_p$ case, we have by Proposition~\ref{prop-polyinterpdet} a lower bound
  $$
  \big|\Delta^D(\bfmath{f},\bfmath{p})\big|\geq|t|^{([F:\QQ_p]/r)\{\log_q(\mu!)+(d+2)D\mu h\}}.
  $$
  On the other hand, by Proposition~\ref{prop-interpdet} we have an
  upper bound
  $$
  \big|\Delta(\bfmath{f},\bfmath{p})\big|\leq \rho^{C_dE^{-1/d}\mu^{1+1/d}}\leq |t|^{\delta C_dE^{-1/d}\mu^{1+1/d}}.
  $$
  Thus
  $$ 
 |t|^{\delta C_dE^{-1/d}\mu^{1+1/d}} \ge |t|^{([F:\QQ_p]/r)\{\log_q(\mu!)+(d+2)D\mu h\}}. 
  $$
  Recalling that $\mu\sim_d D^{d+1}$, we have a contradiction as soon as
  $$
  C'_d r\delta E^{-1/d} D^{d+2+\frac{1}{d}} \ge [F:\QQ_p]\{(d+1)\log_q D+(d+2)C''_d D^{d+2} h\}
  $$
  with suitable constants $C'_d,C''_d$. Thus we obtain a
  contradiction, for instance with a suitable
  $D\sim_d E\cdot (h\cdot [F:\QQ_p]/\varepsilon r)^d$.
  
  In the $F\supset\FF_q\dpl t\dpr$ case, we have
  $$
  |t|^{\delta C_dE^{-1/d}\mu^{1+1/d}} \ge |t|^{(d+2)D\mu h}.
  $$
  Similarly to the previous case, we have a contradiction as soon as
  $$
  C'_d \delta E^{-1/d} D^{d+2+\frac{1}{d}} \ge (d+2)C''_d D^{d+2} h
  $$
  with suitable constants $C'_d,C''_d$, for instance with a suitable
  $D\sim_d E(h/\varepsilon)^d$.
\end{proof}

\subsection{Families of hypersurfaces}\label{sec-ratpthyper}
\subsubsection{The space $\mathcal{H}_n$}\label{sub-spaceHn}
As usual, we identify $\NN^n$ with the set of all monomials in $n$
variables $\bfmath{X}=(X_1,\ldots,X_n)$ by
$\nu=(\nu_1,\ldots,\nu_n)\mapsto\bfmath{X}^{\nu}=X_1^{\nu_1}\cdots X_n^{\nu_n}$,
and let
$$
L(D,n)=\{\nu\in\NN^n\mid |\nu|\leq D\}
$$
denote the set of monomials of degree at most $D$.  Let
$\rR{t}^{L(D,n)}$ be the $\rR{t}$-module of all maps of the form
$L(D,n)\rightarrow\rR{t}$, and consider the $\rR{t}$-scheme
$$
P_{D,n}=\Proj(\Sym \rR{t}^{L(D,n)}),
$$
which is the parameter space that parametrizes all hypersurfaces of
degree at most $D$ in $n$ variables over $\rR{t}$.  Let
$$
\xymatrix{H_{D,n}\ar@{^{(}->}[r]\ar[d]&P_{D,n}\times_{\rR{t}}\Spec\rR{t}[X_1,\ldots,X_n]\ar@/^12pt/[dl]^{\mathrm{pr}_1}\\ P_{D,n}}
$$
be the universal family.

Similarly, we consider the analytic parameter space
$\mathcal{P}_{D,n}=P^{\an}_{D,n}$ and the universal family
$$
\xymatrix{\mathcal{H}_{D,n}\ar@{^{(}->}[r]\ar[d]&\mathcal{P}_{D,n}\times_{\rR{t}}\Sp\rR{t}\dl X_1,\ldots,X_n\dr\ar@/^12pt/[dl]^{\mathrm{pr}_1}\\ \mathcal{P}_{D,n}}.
$$
Note that the fiber of $\mathcal{H}_{D,n}\rightarrow\mathcal{P}_{D,n}$
over a $K$-valued point $\Sp K\rightarrow\mathcal{P}_{D,n}$ is an
affinoid, which is the intersection of the fiber
$(H_{D,n}\times_{P_{D,n}}\Spec K)^{\an}$ and the closed unit-disc
$\Sp K\dl X_1,\ldots,X_n\dr$.

\subsubsection{The space $\mathcal{H}_{D,n,d}$}\label{sub-spaceHnd}
Let $d<n$.  For any subset
$I=\{i_1,\ldots,i_{d+1}\}\subset\{1,\ldots,n\}$ of $d+1$ elements, let
$P^I_{D,d+1}$ (resp.\ $\mathcal{P}^I_{D,d+1}$) be the copy of
$P_{D,d+1}$ (resp.\ $\mathcal{P}_{D,d+1}$) with the coordinates
$\bfmath{X}_I=(X_{i_1},\ldots,X_{i_{d+1}})$.  We set
$$
P_{D,n,d}=\prod_{I}P^I_{D,n,d}\quad (\textrm{resp.}\ \mathcal{P}_{D,n,d}=\prod_{I}\mathcal{P}^I_{D,n,d}),
$$
where the product is taken over all subsets $I\subset\{1,\ldots,n\}$
of $d+1$ elements.  Similarly to the above constructions, we have the
universal families
$$
H_{D,n,d}=\prod_IH^I_{D,d+1}\longrightarrow P_{D,n,d},\quad\mathcal{H}_{D,n,d}=\prod_I\mathcal{H}^I_{D,d+1}\longrightarrow \mathcal{P}_{D,n,d}.
$$

\begin{lem}\label{lem-dimensionatmostd}
  The fibers of $\mathcal{H}_{D,n,d}\rightarrow\mathcal{P}_{D,n,d}$
  over $K$-valued points are of dimension at most $d$.
\end{lem}

\begin{proof}
  It suffices to show that the fibers of
  $H_{D,n,d}\rightarrow P_{D,n,d}$ over closed points of the form
  $\Spec K\rightarrow P_{D,n,d}$, where $K$ is a field, are of
  dimension at most $d$.  This is then completely a question in
  classical algebraic geometry.  Let $\mathcal{F}$ be such a fiber
  over $K$, and suppose that the dimension of $\mathcal{F}$ is $k>d$.
  Then, at a generic point, $\mathcal{F}$ is smooth of dimension $k$,
  and in particular one can choose a subset $I\subset\{1,\ldots,n\}$
  of size $d+1$ such that the projection
  $\bfmath{X}=(X_1,\ldots,X_n)\mapsto\bfmath{X}_I=(X_{i_1},\ldots,X_{i_{d+1}})$
  restricted on $\mathcal{F}$ is smooth on a Zariski open subset.  But
  this contradicts to the fact that the image of this projection is
  contained in a hypersurface defined by a non-zero polynomial of
  degree at most $D$.
\end{proof}

\section{Proof of the counting theorems}
\label{sec-counting-proofs}

In this section we prove the point-counting theorems. In
Section~\ref{sub-mainthm} we formulate and prove a version of the
counting theorem which is uniform in families, for the overconvergent
setting. In Section~\ref{sub-intro-proofs} we deduce from this general
result the various point-counting theorems formulated in the
introduction.

\subsection{The general counting theorem}\label{sub-mainthm}
Let $\bfmath{\dt}=(\dt_1,\ldots,\dt_n)\in(\QQ_{>0})^n$ be a vector of
positive rational numbers, $\mathcal{A}$ an affinoid $\rR{t}$-algebra,
and
$\mathcal{B}=\mathcal{A}\dl\bfmath{x};\bfmath{\dt}\dr/\mathfrak{a}$,
where $\bfmath{x}=(x_1,\ldots,x_n)$ and $\mathfrak{a}$ is an ideal of
$\mathcal{A}\dl\bfmath{x};\bfmath{\dt}\dr$.  Let $A$ be an admissible
formal model of $\mathcal{A}$, and set
$B=A\dl \bfmath{x};\bfmath{\dt}\dr/\udl{\mathfrak{a}}$, where
$\udl{\mathfrak{a}}=\mathfrak{a}\cap A\dl\bfmath{x};\bfmath{\dt}\dr$,
which is an admissible formal model of $\mathcal{B}$.  Thus we have a
family
$$
\varphi\colon\mathcal{X}=\Sp\mathcal{B}\longrightarrow\mathcal{Y}=\Sp\mathcal{A}
$$
of analytic subspaces in the polydisk of polyradius
$|t|^{-\bfmath{\dt}}=(|t|^{-\dt_1},\ldots,|t|^{-\dt_n})$.

Fix $f_1\ldots,f_n\in\cB$ which are algebraic over $\cA[\bfmath{x}]$ and generate $\cB$
over $\cA$, and set $\bfmath{f}:=(f_1,\ldots,f_n)$. 

For a classical point $\beta\colon\Sp F\rightarrow\Sp\rR{t}$ and a classical point $\alpha\colon\Sp K\rightarrow\mathcal{Y}$ over $\beta$, we
denote by $\mathcal{X}_{\alpha}=\Sp\mathcal{B}_{\alpha}$
the fiber of $\varphi$ over $\alpha$.  
For any classical point $\alpha'\colon\Sp K'\rightarrow\mathcal{X}_{\alpha}$ with the corresponding homomorphism $q\colon\mathcal{B}_{\alpha}\rightarrow K'$, where $K'$ is a
finite extension of $F$ containing $K$, such that $q(f_i),q(x_i)\in F$ for $i=1,\ldots,n$, the height $H(\alpha';\bfmath{f})$ is the maximum of the heights of $q(f_i)\in F$ ($i=1,\ldots,n$).
Denote by $\mathcal{X}_{\alpha}(V_F,H;\bfmath{f})$ the set
of all such points $\alpha'$ of the unit-polydisk part
$\udp{(\mathcal{X}_{\alpha})}$ such that $H(\alpha';\bfmath{f})\leq H$,
i.e.,
$$
\mathcal{X}_{\alpha}(V_F,H;\bfmath{f})=\bigg\{\alpha'\colon\Sp
K'\rightarrow\udp{(\mathcal{X}_{\alpha})}\ \bigg| \ \textrm{{\small \begin{minipage}{13em}$q(f_i)$ and $q(x_i)$ lies in $F$ for $i=1,\ldots,n$, and $H(\alpha';\bfmath{f})\leq H$\end{minipage}}}\bigg\}.
$$
Finally, let $\sigma_{\beta}$ be the constant 
$$
\sigma_{\beta}=\begin{cases}[F:\QQ_p]&(F\supset\QQ_p),\\ r_{\beta}&(F\supset\FF_q\dpl t\dpr),\end{cases}
$$
where $r_{\beta}$ is the positive integer such that $t=u\pi^{r_{\beta}}$, with $\pi$ the uniformizer of $V_F$ and $u\in V^{\times}_F$.

\begin{dfn}\label{dfn-algebraic}
  Let $K$ be a non-archimedean local field.

  (1) An irreducible analytic subspace
  $\mathcal{Z}\hookrightarrow\Sp K\dl\bfmath{x};\bfmath{\dt}\dr$ is
  said to be {\em algebraic} if there exists an algebraic subvariety
  $Z\hookrightarrow\Spec K[\bfmath{x}]$ such that
  $\mathcal{Z}^{\cl}\subset Z$ and $\dim\mathcal{Z}=\dim Z$.

  (2) An analytic subspace
  $\mathcal{Z}\hookrightarrow\Sp K\dl\bfmath{x};\bfmath{\dt}\dr$ is
  said to be {\em algebraic} if each of its irreducible components is
  algebraic.
\end{dfn}

\begin{dfn}\label{dfn-algfibers}
  Let
  $\varphi\colon\mathcal{X}=\Sp\mathcal{B}\rightarrow\mathcal{Y}=\Sp\mathcal{A}$
  be a morphism of affinoids over $\rR{t}$ as above.

  (1) The family $\varphi$ is said to have {\em algebraic fibers} if
  every fiber $\mathcal{X}_{\alpha}$ over a classical point of
  $\mathcal{Y}$ is algebraic in the sense of Definition
  \ref{dfn-algebraic} (2).

  (2) The {\em maximal fiber dimension} of the family $\varphi$ is the
  maximum over the dimension of the fibers $\mathcal{X}_{\alpha}$,
  where $\alpha$ ranges over the classical points of $\mathcal{Y}$.
\end{dfn}

Note that under our assumptions on $\bfmath{f}$, being algebraic over
the $\bfmath{f}$ coordinates is equivalent to being algebraic over the
$\bfmath{x}$ coordinates.

\begin{thm}\label{thm-main-family}
  Let $\e>0$.  There exist a transformation diagram of the form
  $$
  \fR{t^{1/N_{\ast}}}\longrightarrow A_{\ast}\longrightarrow A_{\ast}\dl\bfmath{x}_{\ast};\bfmath{\dt}_{\ast}\dr\longrightarrow B_{\ast}=A_{\ast}\dl\bfmath{x}_{\ast};\bfmath{\dt}_{\ast}\dr/\mathfrak{a}_{\ast}\eqno{(\ast)}
  $$
  of admissible $\fR{t}$-algebras modeled on a rooted tree $T$, and for any positive integer $\sigma$, a
  constant $C(B,\e,\sigma)$ satisfying the following conditions:
  \begin{itemize}
  \item[{\rm (a)}] $A_{v_0}\rightarrow B_{v_0}$ at the root $v_0$
    coincides with $A\rightarrow B$;
  \item[{\rm (b)}] for any leaf $w$,
    $\varphi_w\colon\mathcal{X}_w=\Sp\mathcal{B}_w\rightarrow\mathcal{Y}_w=\Sp\mathcal{A}_w$
    has equidimensional algebraic fibers;
  \item[{\rm (c)}] for any classical point
    $\alpha\colon\Sp K\rightarrow\mathcal{Y}$ above $\beta\colon\Sp F\rightarrow\Sp \rR{t}$, we have
    $$
    \mathcal{X}_{\alpha}(V_F,H;\bfmath{f})\subset\bigcup_{(w_j,\alpha_j)}(\mathcal{X}_{w_j})_{\alpha_j}(V_F,H;\bfmath{f}),
    $$
    where $\{(w_j,\alpha_j)\}$ is a collection of
    $C(B,\e,\sigma_{\beta})\cdot H^{\e}$ pairs, where $w_j$ is a leaf of $T$ and
    $\alpha_j$ is a classical point of $\mathcal{Y}_{w_j}$ lying over
    $\alpha$. Here $\bfmath{f}$ on the right hand side denotes the
    image of $\bfmath{f}$ in $(B_{w_j})_{\alpha_j}$.
  \end{itemize}
\end{thm}
\begin{rem}\label{rem:uniformly-algebraic-fibers}
  In fact, the families
  $\varphi_w\colon\mathcal{X}_w=\Sp\mathcal{B}_w\rightarrow\mathcal{Y}_w=\Sp\mathcal{A}_w$
  as in (b) that we will obtain in the following proof is not only of
  algebraic fibers, but is algebraic over $\mathcal{Y}_w$.
\end{rem}

To show the theorem, we need the following proposition, which is the
main inductive step in the proof of the theorem.
\begin{prop}\label{prop-inductivestep1}
  Let $\e>0$.  There exists a transformation diagram of the form
  $(\ast)$ as above satisfying the conditions {\rm (a)}, {\rm (c)},
  and the following {\rm (b$)'$} instead of {\rm (b)}:
  \begin{itemize}
  \item[{\rm (b$)'$}] for any leaf $w$, either
    $\varphi_w\colon\mathcal{X}_w=\Sp\mathcal{B}_w\rightarrow\mathcal{Y}_w=\Sp\mathcal{A}_w$
    has equidimensional algebraic fibers, or the
    maximal fiber dimension of $\varphi_w$ is strictly smaller than
    that of $\varphi$.
  \end{itemize}
\end{prop}

\begin{proof}
  We begin constructing $T$ by forming the stratified Noether
  normalization tree starting from $A\rightarrow B$.  According to
  Proposition \ref{prop-rationaldiagram3}, it will suffice to
  construct a transformation diagram from each of the leaves the
  stratified Noether normalization tree.  Hence we may assume without
  loss of generality that $A\rightarrow B$ is already normalized,
  i.e., there exists a factoring map
  $$
  A\longrightarrow A\dl x_1,\ldots,x_d;\dt\dr\longhookrightarrow B
  $$
  which is finite and injective.
  
  Let $I=\{i_1,\ldots,i_{d+1}\}$ be a subset of size $d+1$ of
  $\{1,\ldots,n\}$, and write
  $\bfmath{f}_I=(f_{i_1},\ldots,f_{i_{d+1}})$.  We apply Proposition
  \ref{prop-hypersurfaces} with $\e/N$ and $\bfmath{f}_I$, where $N$
  is the number of possible $I\subset\{1,\ldots,n\}$, and repeat this
  for all $I\subset\{1,\ldots,n\}$ of size $d+1$, we conclude that the
  set $\mathcal{X}_{\alpha}(V_F,H)$ is contained in at most $H^\e$
  $\bfmath{f}$-pullbacks of fibers of
  $\mathcal{H}_{D,n,d}\rightarrow\mathcal{P}_{D,n,d}$, where $D$ is
  the number given as in Proposition \ref{prop-hypersurfaces}.
  
  More formally, one can replace $\mathcal{P}_{D,n,d}$ by affinoid open subspace that
  contains all the points over which the fibers of $\mathcal{H}_{D,n,d}$
  are those we are considering, and so we may assume that both
  $\mathcal{H}_{D,n,d}$ and $\mathcal{P}_{D,n,d}$ are affinoids.  Let
  $\mathcal{A}'$ be the affinoid algebra such that ($\mathcal{Y}':=$)
  $\Sp\mathcal{A}'=\Sp\mathcal{A}\times_{\rR{t}}\mathcal{P}_{D,n,d}$, and define $\mathcal{X}'=\Sp\cB'$ by the following Cartesian square
  $$
  \xymatrix{\mathcal{X}'\ar@{^{(}->}[r]\ar[d]&\mathcal{P}_{D,n,d}\times_{\rR{t}}\Sp\cB\ar[d]\\ \mathcal{H}_{D,n,d}\ar@{^{(}->}[r]&\mathcal{P}_{D,n,d}\times_{\rR{t}}\Sp\rR{t}\dl X_1,\ldots,X_n\dr\rlap{,}}
  $$
  where the right vertical arrow is induced by $X_i\mapsto f_i$ for $i=1,\ldots,n$.
  By what we have seen
  above, we have
  $$
  \mathcal{X}_{\alpha}(V_F,H;\bfmath{f})\subset\bigcup_j\mathcal{X}'_{\alpha_j}(V_F,H;\bfmath{f}),\eqno{(\dagger)}
  $$
  where $\{\alpha_j\}$ is a collection of $C_n\cdot H^\e$ points of
  $\mathcal{Y}'=\Sp\mathcal{A}'$ lying over $\alpha$, which correspond
  to suitable choices of the hypersurfaces of the form
  $Q(\bfmath{f}_I)=0$ for each $I$. Here $C_n$ is some constant
  depending only on $n$. 
    
  Now we augment out tree by adding a child $v$ of the root, and set
  $A_v=A'$ and $B_v=B'$, where $A'$ and $B'$ are suitable admissible
  formal model of $\mathcal{A}'$ and $\mathcal{B}'$, respectively.  We
  then perform another stratified Noether normalization of
  $A_v\rightarrow B_v$, appending the resulting tree to the root $v$.
  
  Proposition \ref{prop-rationaldiagram3} and $(\dagger)$ ensure that
  the condition (c) in the statement holds.  It remains to verify the
  condition (b$)'$.  Having performed stratified Noether normalization,
  we know that each leaf
  $\varphi_w\colon\mathcal{X}_w\rightarrow\mathcal{Y}_w$ has constant
  fiber dimension $d=d_w$.  If $d_w<d$, we are done.  Otherwise, the
  fibers of $\mathcal{X}_w$ over classical points of $\mathcal{Y}_w$ are
  $d$-equidimensional and contained in a fiber of $\mathcal{H}_{D,n,d}$,
  which is of dimension $d$ due to Lemma \ref{lem-dimensionatmostd}.
  Hence $\varphi_w\colon\mathcal{X}_w\rightarrow\mathcal{Y}_w$ has
  algebraic fibers, as desired.
\end{proof}

\begin{proof}[Proof of Theorem {\rm \ref{thm-main-family}}]
  This is proved by applying Proposition \ref{prop-inductivestep1} to
  the initial family $\varphi\colon\mathcal{X}\rightarrow\mathcal{Y}$,
  and then applying it repeatedly to each of the leaves that have
  non-algebraic fibers.  Since the maximal fiber dimension of the
  fibers drops at each step, this process terminates after at most $n$
  repetitions.
\end{proof}

\subsection{Proofs of the statements from the introduction}
\label{sub-intro-proofs}

\subsubsection{Proof of Theorem~{\rm \ref{thm:main-uniform-oc}}}
\label{proof-main-uniform-oc}

Let $B$ denote a formal model for $\cB$ and recall that
$B=\fR t \dal x_1,\ldots,x_n\dar/\mathfrak{b}$. Set
$B'=\fR{t} \dal t x_1,\ldots, t x_n\dar/\mathfrak{b}'$
where $\mathfrak{b}'$ is the image of $\mathfrak{b}$ by the map
$x_i\to t x_i$. Then $\Sp\gdp\cB 1$ corresponds to
$\udp{(\Sp\cB')}$ under this rescaling. Now Apply
Theorem~\ref{thm-main-family} to $B'$, where we take $A=\fR t$ and
$\bfmath{f}=(t x_1,\ldots, t x_n)$. In the resulting
tree $T$, whenever a leaf $w$ has positive-dimensional fibers it
corresponds by definition to the algebraic part of $\Sp\cB$. The
remaining fibers are zero-dimensional, and in particular their
cardinality is uniformly bounded by some constant depending only on
$B$ (for instance by Remark~\ref{rem:uniformly-algebraic-fibers}). The
conclusion of Theorem~\ref{thm:main-uniform-oc} therefore follows from
the conclusion of Theorem~\ref{thm-main-family}.

\subsubsection{Proof of Theorem~{\rm \ref{thm:main-uniform-family}}}
\label{proof-main-uniform}

Let $B$ denote a formal model for $\cB$ and recall that
$B=A \dal x_1,\ldots,x_n\dar/\mathfrak{b}$. Set
$A'=A \dal y_1,\ldots,y_n\dar$ and
$B'=A'\dal tz_1,\ldots, tz_n\dar/\mathfrak{b}'$ where $\mathfrak{b}'$ is the image of $\mathfrak{b}$
by the map $x_i=y_i+tz_i$.

Fix $\alpha$ and $F_{\alpha}$ as in the statement of
Theorem~\ref{thm:main-uniform-family} and let $V_{\alpha}$ be the valuation
ring of $F_{\alpha}$.  Let $p$ be the morphism $p\colon A\to V_\alpha$
corresponding to $\alpha$. Recall that $q_\alpha$ is the cardinality
of the residue field of $F_\alpha$. Then we may choose a collection of
$q_\alpha^n$ points $\alpha_j\in V_\alpha^n$ corresponding to
morphisms $p_j:A'\to V_\alpha$ with $p_j\vert_A=p$ such that the
union of polydiscs of radius $|t|$ around each $\alpha_j$ covers
$V_\alpha^n$.

Now apply Theorem~\ref{thm-main-family} to $B'$ with $f_i=y_i+tz_i$. Then the
points in $\mathcal{X}_{\alpha_j}(V_\alpha,H;\bfmath{f})$ are in
bijection with the points of $(\Sp\cB)_\alpha(F_{\alpha,0},H)$ belonging to the
ball of radius $|t|$ around $\alpha_j$, and in particular the union
of $q_\alpha^n$ such sets is in bijection with
$(\Sp\cB)_\alpha(F_{\alpha,0},H)$. The conclusion of
Theorem~\ref{thm:main-uniform-family} thus follows from
Theorem~\ref{thm-main-family}.

\frenchspacing
\begin{small}
\bibliographystyle{plain}
\bibliography{refs}

\medskip\noindent {\sc Department of Mathematics, Weizmann Institute of Science, Rehovot, Israel}
(e-mail: {\tt gal.binyamini@weizmann.ac.il})

\medskip\noindent {\sc Department of Mathematics, Tokyo Institute of Technology, 2-12-1 Ookayama, Meguro, Tokyo 152-8551, Japan}
(e-mail: {\tt bungen@math.titech.ac.jp})
\end{small}
\end{document}